\documentclass[11pt]{article}

\usepackage[body={6.5in,9.0in},left=1in,top=1in,centering]{geometry}
\usepackage[final]{pdfpages}
\usepackage[letterpaper=true,colorlinks=true,urlcolor=blue,citecolor=blue,linkcolor=blue,pdfstartview=FitH]{hyperref}

\usepackage{amsmath,mathrsfs,amsfonts,amssymb,graphicx,pstricks}
\usepackage{amssymb,mathrsfs,amsthm,epstopdf,epsfig}
\usepackage{authblk,pstricks}
\newcommand{\specialcell}[1]{\ifmeasuring@#1\else\omit$\displaystyle#1$\ignorespaces\fi}
 \usepackage[sort,comma]{natbib}

\allowdisplaybreaks

\newtheorem{theorem}{Theorem}[section]

\newtheorem{lemma}[theorem]{Lemma}
\newtheorem{prop}[theorem]{Proposition}
\theoremstyle{definition}
\newtheorem{assumption}[theorem]{Assumption}

\newtheorem{rem}[theorem]{Remark}
\newtheorem{example}[theorem]{Example}

\newcommand{\R}{\mathbb R}
\def\ss{\mathbb S}
\def\P{\mathbb P}
\def\d{\mathrm d}
\newcommand{\A}{\mathscr A}
\newcommand{\B}{\mathfrak B}
\newcommand{\E}{\mathbb E}
\newcommand{\LL}{\mathcal L}
\newcommand{\N}{\mathbb N}

\newcommand{\wdt}{\widetilde}
\newcommand{\wdh}{\widehat}
\newcommand{\La}{\Lambda}
\newcommand{\vLa}{\varLambda}
\newcommand{\e}{\varepsilon}
\newcommand{\sgn}{\mathrm{sgn}}
\newcommand{\tr}{\mathrm{tr}}

\newcommand{\abs}[1]{\left\vert #1\right\vert}

\makeatletter \@addtoreset{equation}{section}

\begin{document} 
\title{On Feller and Strong Feller Properties and Irreducibility of Regime-Switching Jump Diffusion Processes with Countable Regimes\thanks{This research was supported in part by the Simons Foundation (grant award number 523736) and a DIG award from the University of Wisconsin-Milwaukee. The  first author was also supported   by   the Development and Promotion of Science and Technology Talents project (DPST).}}
 \author{Khwanchai Kunwai 
   and Chao Zhu\\Department of Mathematical Sciences,   University of Wisconsin-Milwaukee,   Milwaukee, WI 53201,   USA,   {\tt kkunwai@uwm.edu}, {\tt zhu@uwm.edu}}

\maketitle

\begin{abstract}
This work focuses on a class of regime-switching jump diffusion processes with a countably infinite state space for the discrete component.   Such processes can be used to model  complex hybrid systems in which both structural changes,  small fluctuations as well as big spikes coexist and are intertwined. The paper provides  weak sufficient conditions for Feller and strong Feller properties and  irreducibility for such processes.  The  conditions are  presented in terms of the coefficients of the associated stochastic differential equations. 

\medskip	
\noindent{\bf Keywords:}  Regime-witching jump diffusion, Feller property, strong Feller property, irreducibility.
	
\medskip	
\noindent{\bf Mathematics Subject Classification:} 60J27, 60J60, 60J75, 60G51

\end{abstract}

%%%%%%%%%%%%%%%%%%%%%%%%%%%%%%%%%%%%%%%%%%%%%%%%%%%%%%%%%%

\section{Introduction}\label{sect-intro}

 Motivated by the increasing need of modeling complex systems, in which both
structural changes and small fluctuations as well as big spikes coexist and are intertwined,
this paper continues the study on regime-switching jump diffusion processes with countable regimes. 
Our focus is on Feller and strong Feller properties and irreducibility for such processes.  We provide weak sufficient conditions for Feller and strong Feller properties and  irreducibility.   

Roughly speaking,  a regime-switching jump diffusion process can be considered as a two component process $(X(t),\vLa(t))$, an analog  (or continuous state) component $X(t)$ and a switching (or discrete event)   component $\vLa(t)$. The analog component models the state of interest while the switching component can be used to describe the structural changes of the state or random environment or random factors that are not represented by the usual jump diffusion formulation. For instance, a regime-switching Black-Scholes model is considered  in \citet*{Zhang}, in which the continuous component $X(t)$ models the price evolution of a risky asset and the switching component $\vLa(t)$ delineates the overall economy state. Regime-switching jump diffusion is also used in mathematical biology such as the recent paper \citet*{TNDT-19}, in which a stochastic SIRS model subject to  both white   and color noises is analyzed. We refer to  \citet*{MaoY,Shao15-ergodicity,ShaoX-14,Shao-15, YZ-10} and the references therein for more work on regime-switching jump diffusions and their applications. 

In the theory of Markov processes and their applications, dealing with a Markov process $\xi(t)$ with $\xi(0)= x$, for a suitable  function $f$, often one must consider the function $P_{t}f(x): =\E_{x} [f(\xi(t))]$. Following \citet*{Dynkin-I},  the process $\xi(t)$ is said to be {\em Feller} if   $P_{t} f$ is continuous  for any $t\ge 0$ and   $\lim_{t\downarrow 0} P_{t} f(x) =f(x)$ for any bounded and continuous function $f$ and it is said to be {\em strong Feller} if  $P_{t} f$   is continuous for any $t > 0$ and any bounded and measurable function $f$.
This is a natural condition in physical or social modeling: 
a slight perturbation of the initial data should
result in a small perturbation in the subsequent movement. In addition, Feller and strong Feller properties are intrinsically related to the existence and uniqueness of an invariant measure of the underlying process; see, for example, \citet*{MeynT-92,MeynT-93II,MeynT-93III}.  

While Feller and strong Feller properties for regime-switching (jump) diffusion processes have been investigated in the literature, this paper makes substantial improvements over  
the literature. It presents weak local non-Lipschitz conditions for Feller and strong Feller properties.  A standing assumption in the literature (such as \citet*{XiZ-17,Shao-15, YZ-10,NguyenYZ-17}) is that the coefficients of the associated stochastic differential equations are (locally) Lipschitz. While it is a convenient assumption, it is rather restrictive in many applications. For example,   the diffusion coefficients in the Feller branching diffusion  and the Cox-Ingersoll-Ross model  are only H\"older continuous.  For another example,   many control and optimization problems often require the handling of systems where the (local) Lipschitz condition is violated.
Motivated by these considerations, this paper further improves the results in the recent paper \citet*{XiYZ-19} by presenting weak non-Lipschitz conditions for Feller and strong Feller properties. The sufficient conditions are spelled out in Theorems \ref{thm-Feller} and \ref{thm-str-Feller}. While certain technical aspects of the analyses are similar in both papers, the assumptions on the coefficients of the  associated stochastic differential equations in this paper are   substantially weakened; see Remarks \ref{rem-Feller} and \ref{rem-str-Feller} for details.  It is also worth mentioning that the sufficient condition for strong Feller property in Theorem \ref{thm-str-Feller}   is inspired by \citet*{PriolaW-06}, which deals with gradient estimate for diffusion semigroups. The extension from diffusions to  regime-switching  jump diffusions with  countable regimes is   nontrivial as the interactions between the analog and switching components add much subtlety and difficulty to the analyses. 
 
The paper next considers irreducibility of regime-switching jump diffusions.  Irreducibility is a topological property of the underlying stochastic process. Roughly speaking, irreducibility says   that every point   in the state space is reachable from any other point in the state space; see Section \ref{sect-formulation} for the precise definition. Irreducibility plays an important role  in establishing the uniqueness   of an invariant measure for   the underlying Markov process; see, for example, \citet*{Hairer-16,Cerrai-01}. Unfortunately such a property for regime-switching jump diffusions has not been systematically investigated in the literature yet. In this paper, we derive irreducibility for regime-switching jump diffusions (Theorem \ref{thm-irreducibilty}) by using an important identity concerning the transition probability of such processes.   An intermediate step, which is interesting in its own right, is to show that the sub-systems  consists of jump diffusions are irreducible under weaker conditions than those in the recent papers such as \citet*{Qiao-14} and \citet*{XiZ-19}. As an application, we present in Proposition \ref{prop-existence-uniqueness-invariant measure} a set of sufficient conditions under which a unique invariant measure for regime-switching jump diffusions exists. 

The rest of the paper is arranged as follows. We give the precise formulation of regime-switching jump diffusion processes in Section \ref{sect-formulation}. The main results of the paper  are summarized in Section \ref{sect-results}.
  Feller and strong Feller properties for regime-switching jump diffusions are established in Sections \ref{sect-Feller} and \ref{sect-str-Feller}, respectively. Section \ref{sect-irr} derives irreducibility for   regime-switching jump diffusions. Two examples are studied in Section \ref{sect-exms} for demonstration.   Appendix \ref{sect-appendix} contains several technical proofs.

\subsection{Formulation}\label{sect-formulation}
Let $(U,\mathfrak{U})$ be a measurable space, $\nu$ a $\sigma$-finite measure on $U$, and $\mathbb{S} =\{1,2,\dots\}$. Assume further that $d\geq1$ is an integer, $b:\mathbb{R}^{d}\times\mathbb{S}\to  \mathbb{R}^{d}$, $\sigma:\mathbb{R}^{d}\times\mathbb{S}\to  \mathbb{R}^{d\times d}$, and $c:\mathbb{R}^{d}\times\mathbb{S}\times U\to \mathbb{R}^{d}$ are Borel measurable functions. Suppose $(X,\varLambda)$ is a right continuous, strong Markov process with left-hand limits on $\mathbb{R}^{d}\times\mathbb{S}$ such that the first component $X$ satisfies the following stochastic differential equation (SDE),
\begin{equation}\label{eq:X}
dX(t) = b(X(t),\vLa(t))dt + \sigma(X(t),\varLambda(t))dW(t) + \int_{U}c(X(t^{-}),\varLambda(t^{-}),u)\tilde{N}(dt,du),
\end{equation}	
where $W$ is a standard $d$-dimensional Brownian motion,   $N$ is a Poisson random measure on $[0,\infty)\times U$ with intensity $dt\nu(du)$, and $\tilde{N}$ is the associated compensated Poisson random measure. Here the second component $\varLambda$ is supposed to be a continuous-time stochastic process taking values in the set $\mathbb{S}$ and satisfies
\begin{equation}\label{eq-switching-component}
\mathbb{P}\{\varLambda(t+\Delta)=l|\varLambda(t)=k, X(t)=x\} = 
\begin{cases} 
q_{kl}(x)\Delta +o(\Delta) & \text{if } k\neq l \\
1 + q_{kl}(x)\Delta +o(\Delta)     & \text{if } k=l,
\end{cases}
\end{equation}
uniformly in $\mathbb{R}^d$, provided that $\Delta\downarrow 0$.

To obtain the structure of the process $\varLambda$, let us consider the family of disjoint intervals $\{\Delta_{kl}(x) : k, l \in \mathbb{S},x\in \R^{d}\}$ defined on the positive half of the real line as follows:
\begin{eqnarray*}
\Delta_{12}(x) &=& [0,q_{12}(x)),\\
\Delta_{13}(x) &=& [q_{12}(x), q_{12}(x) + q_{13}(x)),\\
&\vdots& \\
\Delta_{21}(x) &=& [q_{1}(x), q_{1}(x) + q_{21}(x)),\\
\Delta_{23}(x) &=& [q_{1}(x) + q_{21}(x), q_{1}(x) + q_{21}(x) + q_{23}(x)),\\
&\vdots& \\
\Delta_{31}(x) &=& [q_{1}(x) + q_{2}(x), q_{1}(x) + q_{2}(x) + q_{31}(x)),\\
&\vdots& 
\end{eqnarray*}
where $q_{k}(x) :=  \sum_{l\in \mathbb{S}\backslash\{k\}}q_{kl}(x)$ and we set $\Delta_{kl}(x) := \emptyset$ if $q_{kl}(x) = 0$ for $k\neq l$. Note that $\{\Delta_{kl}(x) : k, l \in \mathbb{S},x\in \R^{d}\}$ are disjoint intervals and that the length of the interval $\Delta_{kl}(x)$ is equal to $q_{kl}(x)$. Define a function $h:\mathbb{R}^d\times\mathbb{S}\times\mathbb{R}_{+}\to \mathbb{R}$ by
\begin{eqnarray}
h(x,k,r) = \sum_{l\in \mathbb{S}\setminus\{ k\}}1_{\Delta_{kl}(x)}(r).
\end{eqnarray}
In other words, we set 
\begin{align*}
 h(x,k,r) =  
\begin{cases} 
l - k & \text{if } r \in \Delta_{kl}(x) \\
0     & \text{otherwise}
\end{cases}
\end{align*}
for each $x\in \mathbb{R}^d$ and $k \in \mathbb{S}$. As a result, the process $\varLambda$ can be  described  as a solution to the following stochastic differential equation
\begin{eqnarray}\label{eq:La-SDE}
\varLambda(t) = \varLambda(0) + \int_{0}^{t}\int_{\mathbb{R}_{+}}h(X(s^{-}), \varLambda(s^{-}),r)N_{1}(ds,dr),
\end{eqnarray}
where $N_{1}$ is a Poisson random measure on $[0,\infty)\times[0,\infty)$ with characteristic measure $\mathfrak{m}(dz)$, the Lebesgue measure.

We make the following standing assumption throughout the paper:

\begin{assumption}\label{Assum1} 
For any $(x,k) \in \mathbb{R}^d\times\mathbb{S}$, the system of stochastic differential equations \eqref{eq:X} and \eqref{eq:La-SDE} has a non-explosive weak solution $(X^{(x,k)},\varLambda^{(x,k)})$ with initial condition $(x,k)$ and the solution is unique in the sense of probability law. 
\end{assumption}

 Consequently we can consider  the semigroup 
 \begin{equation}\label{eq:swjd-semigroup}
P_{t} f(x,k): = \E_{x,k}[f(X(t), \vLa(t))]= \E[f(X^{(x,k)}(t),\varLambda^{(x,k)}(t))], \quad f\in \B_{b}(\R^{d}\times \ss). 
\end{equation} The main focus of this paper is to investigate the continuity properties of the semigroup $P_{t}$. We say that  the semigroup $P_{t}$ or the process $(X, \vLa)$ is {\em Feller continuous} if $P_{t} f \in C_{b}(\R^{d}\times \ss)$ for % every  $f \in C_{b}(\R^{d}\times \ss)$ and 
all $t\ge 0$ and $\lim_{t\downarrow 0} P_{t} f(x,k) = f(x,k)$ for all $f\in C_{b}(\R^{d}\times \ss)$ and $(x,k) \in \R^{d}\times \ss$. Furthermore, we say that  the semigroup $P_{t}$ or the process $(X, \vLa)$ is {\em strong Feller continuous} if $P_{t} f \in C_{b}(\R^{d}\times \ss)$ for every $f \in \B_{b}(\R^{d}\times \ss)$ and $t > 0$.  

 Denote the transition probability of the process $(X,\varLambda)$ by 
\begin{eqnarray*}
P(t,(x,k),B\times\{l\}) := P_{t} 1_{B\times \{l\}}(x,k) =\mathbb{P}\{(X(t),\varLambda(t)) \in B\times\{l\} | (X(0),\varLambda(0)) = (x,k)\}, 
\end{eqnarray*} for $ B\in \B(\R^{d})$ and $ l\in \ss$.
The semigroup $P_{t}$ of \eqref{eq:swjd-semigroup} is said to be  {\em irreducible} if for any $t > 0$ and $(x,k) \in \mathbb{R}^d \times \mathbb{S}$, we have
$$P(t,(x,k),B\times\{l\}) > 0$$
for all $l \in \mathbb{S}$ and all nonempty open set $B \in \B(\mathbb{R}^d)$.

For convenience, we state the infinitesimal generator of the regime-switching jump diffusion $(X,\varLambda)$  as follows:
\begin{eqnarray}\label{eq:A-operator}
\mathscr A f(x,k) := \mathcal{L}_{k}f(x,k) + Q(x)f(x,k),
\end{eqnarray}
  for $f(\cdot,k)\in C^{2}_{c}(\R^{d})$, where 
\begin{equation}
\label{eq-Lk-generator}
\begin{aligned}
\mathcal{L}_{k}f(x,k) &:= \frac{1}{2}\tr\left(a(x,k)D^{2}f(x,k)\right) + \langle b(x,k),D f(x,k)\rangle  \\
& \quad \ + \int_{U}\left(f(x+c(x,k,u), k) - f(x,k) -\langle D f(x,k), c(x,k,u)\rangle\right)\nu(du),
\end{aligned}
\end{equation}
and 
\begin{equation}
\label{eq-Qx-generator}
\begin{aligned}
Q(x)f(x,k)  := & \sum_{l\in\mathbb{S}}q_{kl}(x)\left[f(x,l) - f(x,k)\right] % \\
 = \int_{[0,\infty)}\left[f(x, k+h(x,k,z)) - f(x,k)\right]\mathfrak{m}(dz).
\end{aligned}
\end{equation} In \eqref{eq-Lk-generator} and throughout the paper, $D f(x,k)$ and $D^{2} f(x,k)$ denote respectively the gradient and Hessian matrix of the function $f$ with respect to the $x$ variable, and $\langle \cdot, \cdot\rangle$ denotes the inner product.  The Hilbert–Schmidt norm of a vector or a matrix $a$ is denoted by  
 $|a|: =\sqrt{ \tr(aa^{T})}$, in which $a^{T}$ is the transpose of $a$.

\subsection{Assumptions and Statements of  Results}\label{sect-results}
We collect the assumptions and the main results in this subsection. 
\subsubsection{Feller Property}
\begin{assumption}\label{assumption-non-lip}
\begin{itemize}
  \item[(i)] 
  If $d=1$, then there exist a positive number $\delta_{0}$ and a nondecreasing and concave function $\rho: [0,\infty)\to  [0,\infty)$ satisfying
\begin{equation}\label{rhoProperties1}
\int_{0^{+}}\frac{dr}{\rho(r)} = \infty,
\end{equation}
such that for all $k\in \mathbb{S}, R>0$ and $x, z \in \mathbb{R}$ with $|x|\vee |z| \leq R$ and $|x-z|\leq \delta_{0}$,
\begin{align}\label{A2.2.2}
& {\sgn}(x-z)(b(x,k) - b(z,k))  \leq \kappa_{R}\rho(|x-z|),
\\
\label{sigma-holder-1/2}
& |\sigma(x,k)-\sigma(z,k)|^2 +\int_{U}|c(x,k,u)-c(z,k,u)|^2\nu(du)   \leq \kappa_{R}|x-z|,
\end{align}
where $\kappa_{R}$ is a positive constant and $\sgn(a)=1_{\{ a> 0\}} -1_{\{ a\leq 0\}}$. In addition, for each $k \in \mathbb{S}$, either
\begin{equation}
 \text{the function }x\mapsto x+c(x,k,u) \text{ is nondecreasing for all } u \in U
\end{equation}
or there exists some $\beta > 0$ such that
\begin{equation}
|x-z +\theta(c(x,k,u)-c(z,k,u))| \geq \beta|x-z|, \forall(x,z,u,\theta) \in \mathbb{R}\times\mathbb{R}\times U\times [0,1].
\end{equation}

  \item[(ii)]  If $d\geq 2$, then there exist a positive number $\delta_{0}$ and a nondecreasing and concave function $\rho: [0,\infty)\to  [0,\infty)$ satisfying
\begin{equation}\label{rhoProperties2}
0 < \rho(r) \leq (1+r)^2\rho(r/(1+r)) \text{ for }  r>0 \text{ and } \int_{0^{+}}\frac{dr}{\rho(r)} = \infty
\end{equation}
so that for all $k\in \mathbb{S}, R>0$ and $x, z \in \mathbb{R}^d$ with $|x|\vee |z| \leq R$ and $|x-z|\leq \delta_{0}$,
\begin{equation}\label{eq-b-sigma-non-lip}\begin{aligned}
2\langle x-z,  b(x,k)-b(z,k)\rangle &  + |\sigma(x,k)-\sigma(z,k)|^2 \\
& + \int_{U}|c(x,k,u)-c(z,k,u)|^2\nu(du)\leq \kappa_{R}\rho(|x-z|^2),
\end{aligned}\end{equation}
where $\kappa_{R}$ is a positive constant. 
\end{itemize}
\end{assumption}

 \begin{assumption}\label{assumption-Q-cont} 
   For each $k \in \ss$, there exists a  concave function  $\gamma_{k}: \R_{+} \mapsto \R_{+}  $ with $\gamma(0) = 0$ such that for all $x, y \in \mathbb{R}^d$ with $|x| \vee |y| \le R$, we have 
\begin{equation}\label{qklH}
\sum_{l\in \mathbb{S}\backslash\{k\}}|q_{kl}(x)-q_{kl}(y)| \leq \kappa_{R} \gamma_{k}(|x-y|).
\end{equation}
for some positive constant $\kappa_{R}$ (which,  without loss of generality, can be assumed to be the same positive constant as in that  \eqref{A2.2.2} and \eqref{sigma-holder-1/2}).
 \end{assumption}
 
 \begin{theorem}\label{thm-Feller}
	Under Assumptions  \ref{assumption-non-lip} and \ref{assumption-Q-cont}, the process $(X,\varLambda)$  has the Feller property.
\end{theorem}

\begin{rem}\label{rem-Feller}
We note that Assumption \ref{assumption-non-lip} is comparable to  the corresponding assumption in \citet*{XiYZ-19}, except that  the non-local term in \eqref{sigma-holder-1/2} and \eqref{eq-b-sigma-non-lip} only requires  the regularity  of $\int_{U}|c(x,k,u)-c(z,k,u)|^2\nu(du)$. In \citet*{XiYZ-19}, the corresponding term is $\int_{U}[|c(x,k,u)-c(z,k,u)|^2 \wedge |x-z|\cdot |c(x,k,u)-c(z,k,u)|]\nu(du)$. 

Assumption \ref{assumption-Q-cont} is weaker than that in \citet*{XiYZ-19}. Indeed,  the paper assumes that $Q(x) = (q_{kl}(x))$ satisfies  \begin{displaymath}
\sum_{l\in \mathbb{S}\backslash\{k\}}|q_{kl}(x)-q_{kl}(y)| \leq \kappa_{R} \,\rho\bigg(\frac{|x-y|}{1+|x-y| }\bigg), \text{\em{ for each }} k\in \ss,
\end{displaymath} for all $x,y\in \R^{d}$ with $|x|\vee |y| \le R$, in which $\kappa_{R} > 0$ and  $\rho$ is an increasing and concave function satisfying \eqref{rhoProperties2}. In contrast, the function $\gamma_{k} $ in Assumption \ref{assumption-Q-cont} may depend on $k$, and is only required to be concave with $\gamma_{k}(0) =0$. In particular, the non-integrability condition $\int_{0^{+}} \frac{dr}{\rho(r)} =\infty$ is dropped.  This relaxation is significant and  renders that the analyses in \citet*{XiYZ-19} are not applicable.
\end{rem}

 \subsubsection{Strong Feller Property}
 
\begin{assumption}\label{Assum3}
   For every $k \in \mathbb{S}$ the following assertions hold:
\begin{itemize}
  \item[(i)] For every $R>0$ there exits a constant $\lambda_{R} > 0$ such that
\begin{eqnarray}\label{eq:elliptic}
\langle \xi, a(x,k)\xi\rangle \geq \lambda_{R}|\xi|^2,  \qquad \xi \in \mathbb{R}^d,
\end{eqnarray} for all $x\in \R^{d}$ with $|x| \le R$, where $a(x,k) := \sigma(x,k)\sigma(x,k)^{T}$. 
 \item[(ii)]There exist a positive constant $\delta_{0}$   and a nonnegative  function  $g\in C (0,\infty) $ satisfying
\begin{eqnarray}\label{g-integrable-01}
\int_{0}^{1}g(r)dr < \infty,
\end{eqnarray}
such that for each $R>0$, there exists a  constant $\kappa_{R}>0$ so that either (a) or (b) below holds:
\begin{itemize}
  \item[(a)] If $d =1$, then  \begin{equation}
\label{eq:1d-str-Fe-coeff-cts}
 2\langle x-z,b(x,k)-b(z,k)\rangle  
  + \int_{U}|c(x,k,u)-c(z,k,u)|^2\nu(du) \leq 2\kappa_{R}|x-z|g(|x-z|),
\end{equation} for all $x, z \in \mathbb{R}$ with $|x|\vee |z| \leq R$ and $|x-z| \leq \delta_{0}$.
  \item[(b)] If $d \ge 2$, then \begin{equation}
\label{eq:str-Fe-coeff-cts}
\begin{aligned}
 |\sigma_{\lambda_{R}}(x,k)& - \sigma_{\lambda_{R}}(z,k)|^2  + 2\langle x-z,b(x,k)-b(z,k)\rangle  \\
&  + \int_{U}|c(x,k,u)-c(z,k,u)|^2\nu(du) \leq 2\kappa_{R}|x-z|g(|x-z|),
\end{aligned}
\end{equation}
for all $x, z \in \mathbb{R}^d$ with $|x|\vee |z| \leq R$ and $|x-z| \leq \delta_{0}$, where  $\sigma_{\lambda_{R}}$ is the unique symmetric nonnegative definite matrix-valued function such that $\sigma_{\lambda_{R}}^2(x,k) = a(x,k) - \lambda_{R}I$. 
\end{itemize}
\end{itemize}
\end{assumption}

\begin{theorem}\label{thm-str-Feller}
Suppose that Assumptions \ref{assumption-Q-cont} and  \ref{Assum3}  hold. Then the process $(X,\varLambda)$ has strong Feller property.
\end{theorem}
\begin{rem}\label{rem-str-Feller}
We remark that Assumption \ref{Assum3} improves significantly over those in the literature such as \citet*{XiZ-17,Shao-15}, which require Lipschitz condition for the coefficients of the associated stochastic differential equations. By contrast, \eqref{eq:1d-str-Fe-coeff-cts} and \eqref{eq:str-Fe-coeff-cts} place very mild conditions on the coefficients. % Consequently,  Assumption \ref{Assum3}
It allows us to treat, for example, the case of H\"older continuous coefficients by taking $g(r) = r^{-p}$ for $0\le p < 1$; see Example \ref{ex1}. For the case when $d =1$, only the drift and the jump coefficients are required to satisfy the regularity conditions. 
 \end{rem}

\subsubsection{Irreducibility}
\begin{assumption} \label{weak solution X^(k)}
		For each $k \in \mathbb{S}$ and $x \in \mathbb{R}^d$, the  SDE   \begin{equation}
\label{SDE X^k} \begin{aligned}X^{(k)}(t) = x+ \int_{0}^{t}b(X^{(k)}(s),  k ) d s +  \int_{0}^{t} \sigma(X^{(k)}(s), k)d W(s)   +  \int_{0}^{t}\int_{U} c(X^{(k)}(s-), k, u)\wdt N(d s, d u)
\end{aligned}\end{equation}   has a non-explosive weak solution $X^{(k)}$ with initial condition $x$ and the solution is unique in the sense of probability law. 
\end{assumption}

\begin{assumption}\label{Assump-linear growth}
For any $x\in \mathbb{R}^d$ and $k \in \mathbb{S}$, we have
\begin{eqnarray}\label{< |x|^2+1}
2\langle x, b(x,k)\rangle   \leq \kappa(|x|^2+1), \quad  \quad  |\sigma(x,k)|^2 + \int_{U}|c(x,k,u)|^2\nu(du) \leq \kappa(|x|^2+1),
\end{eqnarray}	and
\begin{eqnarray}\label{eq1:elliptic}
\langle \xi, a(x,k)\xi\rangle \geq \lambda|\xi|^2,  \qquad \xi \in \mathbb{R}^d,
\end{eqnarray}    where  $\lambda$ and $\kappa$ are positive   constants.
\end{assumption}

\begin{assumption}\label{Assump-Q irreducible}
\begin{itemize}
\item[(i)] There exists a positive constant $\kappa_{0} $ such that 	
	\begin{equation}\label{eq:q_kl-bound}
0 \leq q_{kl}(x) \leq \kappa_{0} l 3^{-l}
\end{equation}
for all $x \in \mathbb{R}^d$ and $k \neq l \in \mathbb{S}$.	
\item[(ii)]  For any $k, l \in \mathbb{S}$, there exist $k_0, k_1,...,k_n \in \mathbb{S}$	with $k_i \neq k_{i+1}, k_0=k$, and $k_n = l$ such that the set $\{x\in \mathbb{R}^d : q_{k_ik_{i+1}}(x) > 0\}$ has positive Lebesgue measure for all $i=0, 1,\dots,n-1$.
\end{itemize}
\end{assumption}

\begin{theorem}\label{thm-irreducibilty}
Suppose that  Assumptions \ref{Assum3} (ii), \ref{weak solution X^(k)},   \ref{Assump-linear growth}, and \ref{Assump-Q irreducible} hold.  Then the semigroup $P_{t}$ of \eqref{eq:swjd-semigroup} is irreducible.
\end{theorem}

\begin{rem}\label{rem-about-irreducibility}
As we mentioned in Remark \ref{rem-str-Feller}, Assumption \ref{Assum3} (ii) places very mild regularity condition for the coefficients of the stochastic differential equations given by \eqref{eq:X} and \eqref{eq:La-SDE}; see Remark \ref{rem-assume-3.1} for further elaborations.  Assumption \ref{weak solution X^(k)} requires each subsystem of \eqref{eq:X} to be well-posed in the weak sense. Assumption \ref{Assum3} (i) is strengthened to uniform ellipticity in \eqref{eq1:elliptic}, which is a common assumption in the literature for deriving irreducibility of (jump) diffusions, see, for example, \citet*{P-Zabczyk-1995,Qiao-14} and others. In addition, since we are dealing with a two-component process $(X,\vLa)$, one can expect that the $q$-matrix $Q(x)$ must satisfy some sort of irreducibility condition so that $(X,\vLa)$ is irreducible; Assumption \ref{Assump-Q irreducible} (ii) is therefore in force. Finally, the linear growth condition \eqref{< |x|^2+1} as well as  \eqref{eq:q_kl-bound} are imposed to facilitate our technical analyses. 
\end{rem}

\section{Feller Property}\label{sect-Feller}
 
This section is devoted to establishing the Feller property for regime-switching jump diffusion stated in Theorem \ref{thm-Feller}. We will use the coupling method to prove Theorem \ref{thm-Feller}. To this end, let us first construct a basic coupling operator $\wdt \A$ for $\A$. For $f(x,i, z,j) \in C_{c}^{2} (\R^{d} \times \ss\times \R^{d} \times \ss )$, we define
\begin{equation}
\label{eq-A-coupling-operator} \begin{aligned}
\wdt{\A} & f(x,i,z,j)  : =\! \bigl[ \wdt \varOmega_{\text{d}}   + \wdt \varOmega_{\text{j}}   + \wdt \varOmega_{\text{s}} \bigr] f(x,i,z,j),
\end{aligned}\end{equation} where $ \wdt \varOmega_{\text{d}}$,  $\wdt \varOmega_{\text{j}}$, and    $ \wdt \varOmega_{\text{s}}$ are defined as follows.
For $x,z\in \R^{d}  $ and $i,j\in \ss $, we set $a(x,i)= \sigma(x,i)\sigma(x,i)'$ and
$$\begin{aligned}a(x,i,z,j) & =\begin{pmatrix}
a(x,i) & \sigma (x,i) \sigma (z,j)' \\
\sigma (z,j) \sigma (x,i)' & a(z,j)
\end{pmatrix},\ \
b(x,i,z,j) =\begin{pmatrix}
b(x,i)\\
b(z,j) \end{pmatrix}. \end{aligned}$$  Then we define
\begin{align}\label{eq-Omega-d-defn}  
& \wdt {\varOmega}_{\text{d}}f(x,i,z,j): =\frac
{1}{2}\hbox{tr}\bigl(a(x,i,z,j)D^{2}f(x,i, z,j)\bigr)  +\langle
b(x,i,z,j), D f(x,i, z,j)\rangle,
 \\
 \label{eq-Omega-j-defn}
& \begin{aligned}    \displaystyle\wdt {\varOmega}_{\text{j}} f(x,i,z,j)
& : = \int_{U}
\big[f(x+c(x,i,u),i, z+c(z,j,u), j)-f(x,i,z,j)  \\
  & \quad   -  \langle D_{x } f(x,i,z,j),   c(x,i,u)\rangle  
 - \langle D_{z} f(x,i, z,j),   c(z,j,u)\rangle   \big] \nu(d u),
\end{aligned}
\end{align} where $D f(x,i, z,j)= (D_{x } f(x,i,z,j), D_{z} f(x,i,z,j))'$ is the gradient and $D^{2}f(x,i,z,j)$ is   the Hessian matrix of   $f$ with respect to the  variables $x$ and $z$,  and
\begin{align}
\label{eq-Q(x)-coupling} 
\nonumber \wdt \varOmega_{\text{s}}  f(x,i,z,j) & : =
 \sum_{l\in\ss}[q_{il}(x)-q_{jl}(z)]^+(   f(x,l, z, j)-  f(x,i, z, j))\\
&\quad\ +\sum_{l\in\ss}[q_{jl}(z) -q_{il}(x)]^+( f(x,i, z, l)- f(x,i, z, j) )\\
\nonumber&\quad \  +\sum_{l\in\ss}[ q_{il}(x) \wedge q_{jl}(z) ](  f(x,l, z, l)-f(x,i, z, j)).
 \end{align}  
For any function $f : \R^{d}\times \R^{d}\mapsto \R$, let $\wdt f:\R^{d} \times \ss\times \R^{d} \times \ss \mapsto \R$ be defined by $\wdt f(x,i,z,j): =f(x,z)$.  Now we   denote   for each $k\in \ss$ $$\wdt \LL_{k} f(x,z) =( \wdt {\varOmega}_{\text{d}}^{(k)} + \wdt {\varOmega}_{\text{j}}^{(k)}   ) f(x,z) : = ( \wdt {\varOmega}_{\text{d}} + \wdt {\varOmega}_{\text{j}}   ) \wdt f(x,k,z,k), \quad \forall f \in C^{2}_{c}(\R^{d}\times \R^{d}). $$

  We introduce the following notations.   Let
$(X(\cdot), \varLambda(\cdot),  \tilde X(\cdot), \tilde \vLa(\cdot))$ denote the coupling process
corresponding to the  operator  $\wdt \A$
with initial condition $(x, k,z, k)$,  in which $\delta_{0} > |x-z | > 0$,  and  $\delta_{0}$ is the positive constant in Assumption \ref{assumption-non-lip}.
  For any $R > 0$, let
\begin{equation}
\label{eq:tau_R-defn}
\tau_{R} := \inf\{t\geq 0: |\tilde{X}(t)|\vee |X(t)|\vee|\tilde{\varLambda}(t)|\vee| \varLambda(t)| > R\}. 
\end{equation} 
In view of Assumption \ref{Assum1},
 $\lim_{R\rightarrow \infty}\tau_{R} = \infty$ a.s. Also denote  $\Delta_{t} = \tilde{X}(t) - X(t)$ and \begin{equation}
\label{eq:S-delta0-defn}
S_{\delta_{0}} := \inf\{t\geq 0 :|\Delta_{t}| > \delta_{0}\}  = \inf\{t\geq 0 :|\tilde{X}(t) - X(t)| > \delta_{0}\}.
\end{equation} Note that $\vLa (0) = \tilde \vLa(0) =k$. We denote by
	 \begin{equation}
\label{eq:zeta-defn}
\zeta := \inf\{t\geq0:\varLambda(t) \neq \tilde{\varLambda}(t)\}
\end{equation}
the first time when the switching components $\vLa$ and $\tilde \vLa$ differ.

We need the following lemma whose proof is arranged in Appendix \ref{sect-appendix}:
% \footnote{Can we obtain the lemma from the assumption \begin{equation}
%\label{eq:Fe-coeff-cts}
%\begin{aligned}
% ||\sigma^2(x,k) - \sigma^2(z,k)||^2 & + 2\langle x-z,b(x,k)-b(z,k)\rangle  \\ &  + \int_{U}|c(x,k,u)-c(z,k,u)|^2v(du) \leq 2|x-z|g(|x-z|),
%\end{aligned}
%\end{equation} for all $x,z\in\R^{d}$ with $|x|\vee |z| \le R$ and $|x-z| \le \delta_{0}$?}

\begin{lemma}\label{lem-24} 
	Under Assumption  %  \ref{assumption-non-linear-growth} and 
	\ref{assumption-non-lip},   
	%  Let $x, \tilde{x} \in \mathbb{R}^d$ and $k\in \ss$. Let   $(X,\varLambda)$ and $(\tilde{X},\tilde{\varLambda})$ be two processes satisfying \eqref{eq:X} and \eqref{eq:La-SDE} with initial conditions $(x,k)$ and $(\tilde x, k)$, respectively.  Denote $\Delta_{t} = \tilde{X}(t) - X(t).$ Define $\tau_{R}$ and $ S_{\delta_{0}}$ as in \eqref{eq:tau_R-defn} and \eqref{eq:S-delta0-defn}, respectively,  and  
	 the following assertion holds: 
		% For any $d \geq 1$ and $x \in \mathbb{R}^d$, the following assertions hold:
	\begin{eqnarray}\label{e1:EDd-=0}
		\lim_{|\tilde{x}-x|\to  0}\mathbb{E}[|\Delta_{t\wedge\tau_{R}\wedge S_{\delta_{0}}\wedge\zeta}|] =  0, \quad \forall t\ge 0.	\end{eqnarray}
		% and hence \begin{equation}\label{EDd-=0}\lim_{|\tilde{x}-x|\to  0}\mathbb{E}[|\Delta_{t\wedge\tau_{R}\wedge S_{\delta_{0}}\wedge\zeta}|^{\delta}] =  0	\end{equation}
\end{lemma}

% Now we are ready to show that the process $(X,\varLambda)$ has Feller property. 

\begin{proof}[Proof of Theorem \ref{thm-Feller}] We need  to show that for each $(x,k) \in \mathbb{R}^d \times \mathbb{S}$ and each $f\in C_{b}(\R^{d}\times \ss)$, the limit $(P_{t}f)(\tilde{x},\tilde{k})\to  (P_{t}f)(x,k)$ as $(\tilde{x},\tilde{k}) \to  (x,k)$ holds for all $t\geq 0$. Since $\mathbb{S} = \{ 1, 2,...\}$ has a discrete topology, it is enough to consider only $(\tilde{x},k) \to  (x,k)$. First, observe that
\begin{align}\label{conclusion0}
|(P_{t}f)(\tilde{x},\tilde{k})- (P_{t}f)(x,k)| &= |\mathbb{E}[f(\tilde{X}(t),\tilde{\varLambda}(t))] - \mathbb{E}[f(X(t),\varLambda(t))]| \nonumber\\
&\leq  |\mathbb{E}[f(\tilde{X}(t),\tilde{\varLambda}(t))] - \mathbb{E}[f(\tilde{X}(t),\varLambda(t))]| \nonumber\\
& \quad +  |\mathbb{E}[f(\tilde{X}(t),\varLambda(t))] -\mathbb{E}[f(X(t),\varLambda(t))]| \nonumber\\
&= |\mathbb{E}[(f(\tilde{X}(t),\tilde{\varLambda}(t)) - f(\tilde{X}(t),\varLambda(t)))1_{\{\zeta \leq t\}}]| \nonumber\\
&\quad+ |\mathbb{E}[(f(\tilde{X}(t),\tilde{\varLambda}(t)) - f(\tilde{X}(t),\varLambda(t)))1_{\{\zeta > t\}}]| \nonumber\\
&\quad + |\mathbb{E}[f(\tilde{X}(t),\varLambda(t))] -\mathbb{E}[f(X(t),\varLambda(t))]| \nonumber\\
&\leq 2||f||_{\infty}\mathbb{P}\{\zeta \leq t\}+ |\mathbb{E}[f(\tilde{X}(t),\varLambda(t))]-\mathbb{E}[f(X(t),\varLambda(t))]|. 
\end{align}
We will show that both terms on the right-hand side of \eqref{conclusion0} converge to 0 as $\tilde x \to x$. 

  Consider the function $\Xi(x,k,z,l): = 1_{\{k\neq l\}}$. It follows directly from the definition that $$\wdt   \A \Xi(x,k,z,l) = \wdt \Omega_{\text{s}} \Xi(x,k,z,l) \le 0, \text{ if } k \neq l.$$   When $k =l$, we have  from \eqref{qklH} that
\begin{align*}% \label{eq-coupling est 1}
   \wdt   \A \Xi(x,k,z,l) & =\wdt \Omega_{\text{s}} \Xi(x,k,z,k) & \\
   &   = \sum_{ i\in \ss} [q_{ki}(x)-q_{ki}(z)]^+( 1_{\{i \neq k \}} - 1_{\{ k\neq k\}})% \\\nonumber  &\qquad
 +\sum_{i \in\ss}[q_{ki}(z)  -q_{ki}(x)]^+(  1_{\{i \neq k \}} - 1_{\{ k\neq k\}} ) \\
 & \le \sum_{i \in\ss, i \neq k} \abs{q_{ki}(x)-q_{ki}(z)}
 \le  \kappa_{R} \gamma_{k}( \abs{x-y}).
\end{align*} % Hence \eqref{eq1-switching-est} holds 
Hence \begin{equation}
\label{eq-coupling est 1}
\wdt \A \Xi(x,k,z,l) \le   { \kappa_{R} \gamma_{k}( \abs{x-y})}
\end{equation}
for all $k,l \in\ss$ and $x,z\in \R^{d}$ with  $|x| \vee |z| \le R$.

 Note that $\zeta \leq t \wedge \tau_{R}\wedge S_{\delta_{0}}$ if and only if  $\tilde{\varLambda}(t \wedge \tau_{R}\wedge S_{\delta_{0}}\wedge \zeta) \neq  \varLambda(t \wedge \tau_{R}\wedge S_{\delta_{0}}\wedge \zeta)$. Thus  we can use \eqref{eq-coupling est 1} to  compute
  % \footnote{I have trouble to see why the first inequality below holds?}
\begin{align*}
	\mathbb{P} & \{\zeta \leq t \wedge \tau_{R}\wedge S_{\delta_{0}} \}\\ 
	& = \E[\Xi(\tilde X(t \wedge \tau_{R}\wedge S_{\delta_{0}}\wedge \zeta), \tilde \vLa(t \wedge \tau_{R}\wedge S_{\delta_{0}}\wedge \zeta), X(t \wedge \tau_{R}\wedge S_{\delta_{0}}\wedge \zeta),\La(t \wedge \tau_{R}\wedge S_{\delta_{0}}\wedge \zeta))]\\
	%  &= \mathbb{E}[1_{\{\tilde{\varLambda}(t \wedge \tau_{R}\wedge S_{\delta_{0}}\wedge \zeta)- \varLambda(t \wedge \tau_{R}\wedge S_{\delta_{0}}\wedge \zeta) \neq 0\}}]\\
	&=\Xi(\tilde x,k,x,k) + \mathbb{E}\bigg[\int_{0}^{t \wedge \tau_{R}\wedge S_{\delta_{0}}\wedge\zeta}\wdt \A \Xi(\tilde X(s),\tilde\vLa(s), X(s),\La(s)) ds\bigg]\\
	% &= \mathbb{E}\bigg[\int_{0}^{t \wedge \tau_{R}\wedge S_{\delta_{0}}\wedge\zeta}\int_{\mathbb{R}^+}1_{\{ h(\tilde{X}(s^-),\varLambda(s^-),z) - h(X(s^-),\varLambda(s^-),z)\neq 0\}}m(dz)ds\bigg]\\
	% &\leq \mathbb{E}\bigg[\int_{0}^{ \wedge \tau_{R}\wedge S_{\delta_{0}}t\wedge\zeta}  \sum_{l\in\mathbb{S}, l\neq \varLambda(s^-)}|q_{\varLambda(s^-),l}(\tilde{X}(s^-)) - q_{\varLambda(s^-),l}(X(s^-))| ds\bigg]\\
	&\leq \kappa_{R}\mathbb{E}\left[\int_{0}^{t\wedge \tau_{R}\wedge S_{\delta_{0}}\wedge\zeta }{ \gamma_{k} (|\tilde{X}(s)-X(s)|)  } ds\right]\\
	&\leq \kappa_{R}\int_{0}^{t} \mathbb{E}[\gamma_{k}(|\tilde{X}(s \wedge \tau_{R}\wedge S_{\delta_{0}}\wedge\zeta)-X(s \wedge \tau_{R}\wedge S_{\delta_{0}}\wedge\zeta)|)  ] ds \\
	&\le   \kappa_{R}\int_{0}^{t}{  \gamma_{k}\big( \mathbb{E}[|\Delta_{s\wedge\tau_{R}\wedge S_{\delta_{0}}\wedge \zeta}|]\big)  } ds,
\end{align*}
where the last inequality follows from the concavity of $\gamma_{k}$.
Since $\gamma_{k}(0) = 0$, then 
\eqref{e1:EDd-=0} and the bounded convergence theorem imply
\begin{equation}\label{p1=0}
\lim_{|\wdt x-x| \to 0}\P\{\zeta\le t \wedge \tau_{R}\wedge S_{\delta_{0}} \}=0.
\end{equation}

   Note also that on the set $\{S_{\delta_{0}} \leq t\wedge \zeta\wedge \tau_{R}\}$ we have $\delta_{0} \leq |\Delta_{S_{\delta_{0}}\wedge t \wedge\zeta\wedge \tau_{R}}|$.
 % Since $H$ is increasing on the interval $(0,\infty)$, we have $$0 < H(\delta_{0}) \leq H(|\Delta_{t\wedge S_{\delta_{0}}\wedge\zeta\wedge \tau_{R}}|).$$
 This implies
 $$\delta_{0}\mathbb{P}\{S_{\delta_{0}} \leq t\wedge\zeta\wedge \tau_{R}\} \leq  \E[ |\Delta_{t\wedge S_{\delta_{0}}\wedge\zeta\wedge \tau_{R}}|1_{\{S_{\delta_{0}} \leq t\wedge\zeta\wedge \tau_{R}\}}] \leq \mathbb{E}[|\Delta_{t\wedge S_{\delta_{0}}\wedge\zeta\wedge \tau_{R}}|] .$$
Therefore, it follows from  \eqref{e1:EDd-=0} that
 \begin{equation}\label{p2=0}
 \lim_{|\wdt x-x| \to 0}\mathbb{P}\{S_{\delta_{0}} \leq t\wedge\zeta\wedge \tau_{R}\} = 0.
 \end{equation}
 %  Consider the metric $d$ on $\R^{d}\times \ss$ defined by $d((x,i),(y,j)): = |x-y| + 1_{\{i\neq j \}}$. One can verify that $d$ is indeed a metric on $\R^{d}\times \ss$. If we can show that $\tilde{X}(t)\to X(t)$ in probability as $\tilde{x}\to x$, it follows that $(\tilde X(t),\Lambda (t)) \to (X(t), \Lambda(t))$ in probability as $\tilde x \to x$. Because the function $f$ is continuous, we also have $f(\tilde X(t),\Lambda (t)) \to f(X(t), \Lambda(t))$ in probability as $\tilde x \to x$. Then   the bounded convergence theorem implies  %  (\ref{conclsusion1}).
% \begin{eqnarray}\label{conclsusion1} |\mathbb{E}[f(\tilde{X}(t),\varLambda(t))] -\mathbb{E}[f(X(t),\varLambda(t))]|\to  0  \text{ as } \tilde{x}\to  x. \end{eqnarray}
%   Now we show that (\ref{conclsusion1}) holds. 
Fix  an arbitrary  positive number $\epsilon $. We have from \eqref{e1:EDd-=0}  that
%  \begin{eqnarray*}
 %	\mathbb{P}\{|\Delta_{t \wedge S_{\delta_{0}} \wedge \tau_{R}\wedge\zeta }| > \epsilon\} 
%	 % &=& \mathbb{P}\{H(|\Delta_{t \wedge S_{\delta_{0}} \wedge \tau_{R}\wedge\zeta}|) > H(\epsilon)\}\\
 %	&\leq& \frac{\mathbb{E}\big[|\Delta_{t \wedge S_{\delta_{0}} \wedge\tau_{R}\wedge\zeta}|\big]}{\epsilon}.
 % \end{eqnarray*}
%   From (\ref{EHd2=0}), we have
 \begin{eqnarray}\label{p3=0}
 \lim_{|\wdt x-x| \to 0}\mathbb{P}\{|\Delta_{t \wedge S_{\delta_{0}} \wedge \tau_{R}\wedge\zeta }| > \epsilon\} \le  \lim_{|\wdt x-x| \to 0}  \frac{\mathbb{E}\big[|\Delta_{t \wedge S_{\delta_{0}} \wedge\tau_{R}\wedge\zeta}|\big]}{\epsilon} = 0.
 \end{eqnarray}
 
  Since $\lim_{R\rightarrow \infty}\tau_{R} = \infty$ a.s., we can choose $R$ sufficiently large so that \begin{equation}
\label{eq-tau-R<t}
\mathbb{P}\{\tau_{R} < t\} < \epsilon.
\end{equation} Then
  \begin{align*}
  \P\{|\Delta_{t} | > \e \}  & =  \P\{|\Delta_{t} | > \e, \tau_{R} <  t \} +   \P\{|\Delta_{t} | > \e, \tau_{R} \ge  t,\zeta \le t \wedge S_{\delta_{0}} \wedge \tau_{R} \} \\ 
  & \quad +  \P\{|\Delta_{t} | > \e, \tau_{R} \ge  t,\zeta >  t \wedge S_{\delta_{0}} \wedge \tau_{R}, S_{\delta_{0}} \le t \wedge \tau_{R} \wedge \zeta \} \\
  & \quad + \P\{|\Delta_{t} | > \e, \tau_{R} \ge  t,\zeta >  t \wedge S_{\delta_{0}} \wedge \tau_{R}, S_{\delta_{0}} > t \wedge \tau_{R} \wedge \zeta \} \\
  & \le \epsilon+ \P \{\zeta \le t \wedge S_{\delta_{0}} \wedge \tau_{R} \} + \P \{ S_{\delta_{0}} \le t \wedge \tau_{R}\wedge \zeta\} +  \P\{|\Delta_{t} | > \e,  t \le  S_{\delta_{0}} \wedge \tau_{R} \wedge \zeta \} \\
  & \le \epsilon + \P \{\zeta \le t \wedge S_{\delta_{0}} \wedge \tau_{R} \} + \P \{ S_{\delta_{0}} \le t \wedge \tau_{R}\wedge \zeta\} +\P\{|\Delta_{t \wedge S_{\delta_{0}} \wedge \tau_{R} \wedge \zeta} | > \e\}.
  \end{align*}   
  From  \eqref{p1=0}--\eqref{p3=0} we have $$\lim_{|\wdt x-x| \to 0}\mathbb{P}\{|\Delta_{t} | > \e \} \leq \epsilon.$$ Since $\epsilon$ is arbitrary, we conclude that $\lim_{|\wdt x-x| \to 0}\mathbb{P}\{|\Delta_{t} | > \e \} = 0$. In other words, $\tilde{X}(t)\to X(t)$ in probability as $\tilde{x}\to x$. With the   metric $d$ on $\R^{d}\times \ss$ defined by $d((x,i),(y,j)): = |x-y| + 1_{\{i\neq j \}}$, we see immediately that  $(\tilde X(t),\Lambda (t)) \to (X(t), \Lambda(t))$ in probability as $\tilde x \to x$. Because the function $f$ is continuous, we also have $f(\tilde X(t),\Lambda (t)) \to f(X(t), \Lambda(t))$ in probability as $\tilde x \to x$. Then   the bounded convergence theorem implies  %  (\ref{conclsusion1}).
\begin{eqnarray}\label{conclsusion1}
|\mathbb{E}[f(\tilde{X}(t),\varLambda(t))] -\mathbb{E}[f(X(t),\varLambda(t))]|\to  0  \text{ as } \tilde{x}\to  x.
\end{eqnarray}
% . and therefore (\ref{conclsusion1}) follows.\\
  
  Next, we show that $\lim_{\tilde x \to x}\mathbb{P}\{\zeta \leq t\} = 0$ holds.     Thanks to \eqref{eq-tau-R<t} we can compute
 \begin{align*}
 \mathbb{P}\{\zeta \leq t\} &= \mathbb{P}\{\zeta \leq t, \tau_{R}< t\} + \mathbb{P}\{\zeta \leq t, \tau_{R} \geq t\}\\
 &\le  \mathbb{P}\{\tau_{R}< t\} + \mathbb{P}\{\zeta \leq t, \tau_{R} \geq t, S_{\delta_{0}} \leq t\wedge\zeta\} + \mathbb{P}\{\zeta \leq t, \tau_{R} \geq t, S_{\delta_{0}} > t\wedge\zeta\}\\
  % &\leq&  \epsilon +\mathbb{P}\{\zeta \leq t, \tau_{R} \geq t, S_{\delta_{0}} \leq t\wedge\zeta\} + \mathbb{P}\{\zeta \leq t, \tau_{R} \geq t, S_{\delta_{0}} > t\wedge\zeta\}\\
  &\leq  \epsilon + \mathbb{P}\{S_{\delta_{0}} \leq t\wedge\zeta\wedge
  \tau_{R}\} + \mathbb{P}\{\zeta \leq t\wedge \tau_{R} \wedge S_{\delta_{0}}\}.
 \end{align*}
 It then  follows from (\ref{p1=0}) and (\ref{p2=0}) that $\lim_{|\wdt x-x| \to 0}\mathbb{P}\{\zeta \leq t\} \leq \epsilon$. Again since $\epsilon$ is arbitrary,  we have
\begin{eqnarray}\label{conclusion2}
\lim_{|\wdt x-x| \to 0}\mathbb{P}\{\zeta \leq t\} = 0.
\end{eqnarray}

 % In view of (\ref{conclusion0}), (\ref{conclsusion1}) together with (\ref{conclusion2}) 
 Finally we plug  (\ref{conclsusion1}) and  (\ref{conclusion2}) into (\ref{conclusion0}) to complete  the proof.
\end{proof}

\section{Strong Feller Property}\label{sect-str-Feller} 
 As in Section \ref{sect-Feller}, we will use the coupling method to prove Theorem \ref{thm-str-Feller}. To this end,  we first  define 
\begin{displaymath}
\wdh a(x,i,z,j) := \bordermatrix{&  & \cr
	& a(x,i) & \hat{g}(x,i,z,j) \cr
	& \hat{g}(x,i,z,j)^{T} & a(z,j) \cr}\quad  \text{ and } \quad b(x,i,z,j) := \bordermatrix{& \cr
		 & b(x,i)\cr
		 & b(z,j)\cr}
\end{displaymath}
where
\begin{align*}
\hat{g}(x,i,z,j) &:= \lambda_{R}(I-2u(x,z)u(x,z)^T) + \sigma_{\lambda_{R}}(x,i)\sigma_{\lambda_{R}}(z,j)^{T}, % \\
% u(x,z) &:= \frac{x-z}{|x-z|}.
\end{align*} and $u(x,z) := \frac{x-z}{|x-z|}.$
  Then we define the coupling operator $\wdh \A$ for $\A$ of \eqref{eq:A-operator} as follows:
 \begin{equation}\label{eq:Ahat-operator-defn}
 \wdh\A f(x,i,z,j) := [\wdh{\varOmega}_{\mathrm d} + \wdt{\varOmega}_{\mathrm j} + \wdt{\varOmega}_{\mathrm s}]f(x,i,z,j), \quad f \in C^{2}_{c}(\R^{d}\times \ss \times \R^{d}\times \ss), 
 \end{equation}
 where
 \begin{equation}\label{eq:Omega_d-hat-defn}
  \wdh{\varOmega}_{\mathrm d}f(x,i,z,j) = \frac{1}{2}\tr\left(\wdh a(x,i,z,j)D^{2}f(x,i,z,j)\right) + \langle b(x,i,z,j),D f(x,i,z,j)\rangle, 
  \end{equation}
 % where $Df(x,i,z,j) = (D_x f(x,i,z,j),D_z f(x,i,z,j))^T$ is the gradient of $f$ and $D^2 f(x,i,z,j)$ is the Hessian matrix of $f$ with respect to $x$ and $z$.
 and $\wdt{\varOmega}_{\mathrm j} $ and $\wdt{\varOmega}_{\mathrm s}$ are defined as in \eqref{eq-Omega-j-defn} and \eqref{eq-Q(x)-coupling}, respectively. In addition,   as in Section \ref{sect-Feller}, for each $k\in \ss$ and any $F\in C^{2}_{c}(\R^{d}\times \R^{d})$, we write  $f(x,k,z,k): = F(x,z)$ and denote \begin{equation}\label{eq-hatLk-defn} 
\wdh\LL_{k} F(x,z) = [\wdh{\varOmega}_{\mathrm d}^{(k)} + \wdt{\varOmega}_{\mathrm j}^{(k)}]   f(x,k,z,k)   : =  \wdh \A f(x,k,z,k). \end{equation} 
   Note that $\wdh\LL_{k}$ is a   coupling operator for  $\LL_{k}$ defined in \eqref{eq-Lk-generator}.

% \noindent \textbf{NOTATIONS} Denote by $\langle ~,~ \rangle$ the ordinary inner product in $\mathbb{R}^d$. We also note that $\langle \xi, \eta \rangle = \xi^T\eta$ for  $\xi, \eta \in \mathbb{R}^d$. Moreover, we set $||A|| := \sqrt{tr(AA^T)}$ for any matrix $A$, where $trA$ denotes the trace of the matrix. 
Furthermore, to facilitate future presentations, we introduce the following notations.  For any $x, z \in \mathbb{R}^d$ and $i, j \in \mathbb{S}$,   we let	
		\begin{align*}
			A(x,i,z,j) &:= a(x,i) + a(z,j) - 2\hat{g}(x,i,z,j),\\
			\bar{A}(x,i,z,j) &:= \frac{1}{|x-z|^2}\langle x-z , A(x,i,z,j)(x-z)\rangle,\\
			B(x,i,z,j) &:= \langle x-z, b(x,i) - b(z,j)\rangle. 
		\end{align*}

  \begin{lemma}
  For all  $x, z \in \mathbb{R}^d$ and $i, j \in \mathbb{S}$, we have	
\begin{itemize}\parskip=2pt
  \item[{\em (i)}]   $\wdh a(x,i,z,j)$ is symmetric and uniformly positive definite,
\item[{\em (ii)}]  $\tr A(x,i,z,j) = |\sigma_{\lambda_{R}}(x,i) - \sigma_{\lambda_{R}}(z,j)|^2 + 4\lambda_{R}$,  and
\item[{\em (iii)}] $\bar{A}(x,i,z,j) \geq 4\lambda_{R}$. 	
\end{itemize}\end{lemma}
	\begin{proof}
The proof involves  elementary and straightforward computations; similar computations can be found in \citet*{ChenLi-89} and \citet*{PriolaW-06}. We shall omit the details here.  
\end{proof}

%%%%%%%%%%%%%%%%%%%%%%%%%%%%%%%%%%%%%%%%%%%%%%%%%%%%%

% We will denote by $\tilde{\mathcal{L}}$ the infinitesimal generator of the regime-switch jump diffusion process corresponding to the diffusion matrix $a(x,i,z,j)$ and the drift coefficient matrix $b(x,i,z,j)$. 

 Now, let $\phi \in C^2([0,\infty))$. As  in \citet*{ChenLi-89}, for each $k\in \ss$ and all $x, z \in \mathbb{R}^d$ with $x \neq z$, we can verify that
 \begin{equation}\label{Omega_d}\begin{aligned}
 \wdh{\varOmega}_{\mathrm d}^{(k)}\phi(|x-z|) &= \frac{\phi''(|x-z|)}{2}\bar{A}(x,k,z,k) \\&\quad  + \frac{\phi'(|x-z|)}{2|x-z|}[\tr A(x,k,z,k) 
  - \bar{A}(x,k,z,k) +2B(x,k,z,k)].
 \end{aligned}\end{equation}
 Moreover, we have
  \begin{equation}\label{Omega_j}\begin{aligned}
 \wdt{\varOmega}_{\mathrm j}^{(k)} \phi(|x-z|) &= \int_{U}(\phi(|x+c(x,k,u)- z-c(z,k,u)|) - \phi(|x-z|) \\
 &\qquad -\frac{\phi'(|x-z|)}{|x-z|}\langle x-z, c(x,k,u) -c(z,k,u)\rangle)\nu(du).
 \end{aligned}\end{equation}
%  and  \begin{eqnarray}\label{Omega_s} \tilde{\varOmega}_s\phi (|x-z|) = 0.  \end{eqnarray} Therefore,\begin{eqnarray}\label{L=L_k} \tilde{\mathcal{L}}\phi(|x-z|) = [\tilde{\varOmega}_d + \tilde{\varOmega}_j]\phi(|x-z|).\end{eqnarray}
 %%%%%%%%%%%%%%%%%%%%%%%%%%%%%%%%%%%%%%%%%%%%%%%%%%%%%
 
 Motivated by % the argument used in
  \citet*{PriolaW-06}, we consider the function $G$  given by
 \begin{eqnarray*}
G(r):= \int_{0}^{r}\exp\bigg\{-\int_{0}^{s}\frac{\kappa_{R}}{2\lambda_{R}} g(w) dw\bigg\} \int_{s}^{1}\exp \bigg\{\int_{0}^{v}\frac{\kappa_{R}}{2\lambda_{R}} g(u) du\bigg\}dv ds, \qquad r \in [0,1],
 \end{eqnarray*} 
where $g$ is the function given in Assumption \ref{Assum3} (ii). 
 %  in \eqref{g-integrable-01} and \eqref{eq:str-Fe-coeff-cts}. 
 Since $g \geq 0$, we see that
 \begin{align}\label{S'>0, S''<0}
 G'(r) = e^{-\int_{0}^{r}\frac{\kappa_{R}}{2\lambda_{R}}g(w)dw}\int_{r}^{1}e^{\int_{0}^{v}\frac{\kappa_{R}}{2\lambda_{R}}g(u)du}dv\geq 0,
 \text{~~and~~}
 G''(r) = -1 - \frac{\kappa_{R}}{2\lambda_{R}}g(r)G'(r) \leq 0. 
 \end{align}
Note also that $G$ is concave and $\lim_{r\rightarrow 0}G(r) = 0$. Since $G'(0) \geq 1$ and $G(0)=0$, there exists a   constant $\alpha \in (0,1)$ so that
\begin{equation}\label{r < S(r)}
r \leq G(r)  \quad \text{ for all  }\  r \in [0,\alpha].
\end{equation}

\begin{lemma}\label{lem-LkG-estimate}
Suppose  that Assumptions \ref{Assum3} holds. Then for any $R>0$ and $k \in \mathbb{S}$ there exits a positive constant $\beta_{R} > 0$ such that   % \footnote{Note that $\beta= 2 \lambda_{R}$ depends on $R$.}
\begin{eqnarray}\label{LS<beta}
\wdh\LL_{k}G(|x-z|) \leq -\beta_{R}
\end{eqnarray}
for all $x, z \in \mathbb{R}^d$ with $|z|\vee|x| \leq R$ and $0 < |x-z|\leq \alpha\wedge \delta_{0}$, where $\alpha > 0$ is given in \eqref{r < S(r)}.
\end{lemma}	This lemma follows directly from  straightforward but involved   computations. To preserve the flow of reading, we arrange it to Appendix \ref{sect-appendix}.

 Throughout the rest of the section, we use the following notations.  For any   $x,\tilde x \in \mathbb{R}^d$ and $k \in \mathbb{S}$, denote by  $(X(\cdot),\varLambda(\cdot),\tilde{X}(\cdot),\tilde{\varLambda}(\cdot))$  the  process corresponding to the coupling operator $\wdh \A$ with initial condition $(x,k,\tilde{x},k)$.   As in Section \ref{sect-Feller}, denote $\Delta_{t}: = \tilde X(t) - X(t)$ for $t\ge 0$. Let $\tau_{R}$,  $ S_{\delta_{0}}$, and $\zeta$ be defined  as in \eqref{eq:tau_R-defn},  \eqref{eq:S-delta0-defn}, and \eqref{eq:zeta-defn}, respectively.	In addition, for each $n\in \N$, we define
\begin{equation}\label{eq-Tn-defn}
T_n := \inf\bigg\{t\geq 0: |X(t) - \tilde{X}(t)| < \frac{1}{n} \bigg\}.
\end{equation}
Then $\lim_{n\rightarrow\infty}T_n = T$, where
\begin{equation}\label{eq-T-defn}
T :=\inf\{t\geq 0: X(t) = \tilde{X}(t)\}.
\end{equation}

\begin{lemma}
Suppose Assumption \ref{Assum3} holds.  Then the following assertions hold for every $t\geq 0 $:
\begin{align}\label{limE[S(D)]=0}
	& \lim\limits_{|\tilde{x}-x|\to  0}\mathbb{E}[G(|\Delta_{t\wedge\tau_{R}\wedge S_{\bar{\delta}}\wedge\zeta}|)] =  0, \text{ and }
\\ \label{e2:limE[S(D)]=0}	& \lim\limits_{|\tilde{x}-x|\to  0}\mathbb{E}[G(|\Delta_{t\wedge\tau_{R}\wedge S_{\bar{\delta}}\wedge\zeta^-}|)] =  0.
\end{align}
In particular,
\begin{align}
\label{e1:limE[Del^d =0}	& \lim\limits_{|\tilde{x}-x|\to  0}\mathbb{E}[|\Delta_{t\wedge\tau_{R}\wedge S_{\bar{\delta}}\wedge\zeta^{-}}|] =  0,
 % \\ \label{limE[Del^d =0}
 %	& \lim\limits_{|\tilde{x}-x|\to  0}\mathbb{E}[|\Delta_{t\wedge\tau_{R}\wedge S_{\bar{\delta}}\wedge \zeta^{-}}|^{\delta}] = 0
\end{align}
where $\bar{\delta}:= \delta_{0}\wedge\alpha$,   $\delta_{ 0} $ is the constant given in Assumption \ref{Assum3} (ii), and $\alpha \in (0,1) $ is the constant given in \eqref{r < S(r)}.
 % \eqref{Ass sum|q-q| < K|x-y|}.
\end{lemma}

\begin{proof} %  Given  $x,\tilde x \in \mathbb{R}^d$ and $k \in \mathbb{S}$. Let $\tilde{x} \in \mathbb{R}^d$ be such that $\bar{\delta} \geq |x-\tilde{x}| > 0$. Let $(X(t),\varLambda(t),\tilde{X}(t),\tilde{\varLambda}(t))$ denote the coupling process corresponding to the coupling operator $\tilde{\mathcal{L}}$ with initial condition $(x,k,\tilde{x},k)$. It follows from (\ref{LS<beta}) and the Ito's formula that
Assume without loss of generality that $\bar{\delta} \geq |x-\tilde{x}| > 0$. We apply It\^o's formula to the process $G(|\tilde X(\cdot)- X(\cdot)|) = G(|\Delta_{\cdot}|)$: 
\begin{align*}
\mathbb{E}[G(|\Delta_{t\wedge\tau_{R}\wedge S_{\bar{\delta}}\wedge\zeta}|)] &= G(|\Delta_0|) + \mathbb{E}\bigg[\int_{0}^{t\wedge\tau_{R}\wedge S_{\bar{\delta}}\wedge\zeta}\wdh{\mathcal{L}}G(|\Delta_s|)ds\bigg] \\
&\leq G(|\Delta_0|) - \beta_{R}\mathbb{E}[t\wedge\tau_{R}\wedge S_{\bar{\delta}}\wedge\zeta],
\end{align*} where the last inequality follows from (\ref{LS<beta}). 
Hence
\begin{align*}
\mathbb{E}[G(|\Delta_{t\wedge\tau_{R}\wedge S_{\bar{\delta}}\wedge\zeta}|)] + \beta_{R}\mathbb{E}[t\wedge\tau_{R}\wedge S_{\bar{\delta}}\wedge\zeta] \leq G(|\Delta_0|) = G(|x-\tilde{x}|).
\end{align*}
Since $\lim_{r\to 0}G(r) = 0$,    \eqref{limE[S(D)]=0} follows. The same argument implies \eqref{e2:limE[S(D)]=0}. Since $|\Delta_{t\wedge\tau_{R}\wedge S_{\bar{\delta}}\wedge\zeta^-}| \leq \bar{\delta} \leq \alpha$, 
it follows from (\ref{r < S(r)}) that
\begin{eqnarray*}
|\Delta_{t\wedge\tau_{R}\wedge S_{\bar{\delta}}\wedge\zeta^-}| \leq G(|\Delta_{t\wedge\tau_{R}\wedge S_{\bar{\delta}}\wedge\zeta^-}|)
\end{eqnarray*}
and therefore \eqref{e1:limE[Del^d =0} follows as well. % Furthermore, the Holder inequality gives \eqref{limE[Del^d =0}.
% \begin{eqnarray*}	\lim\limits_{|\tilde{x}-x|\to  0}\mathbb{E}[|\Delta_{t\wedge\tau_{R}\wedge S_{\bar{\delta}}\wedge\zeta^-}|] =  0. \end{eqnarray*}
% Then the Holder's inequality gives
%\begin{eqnarray*}	\lim\limits_{|\tilde{x}-x|\to  0}\mathbb{E}[|\Delta_{t\wedge\tau_{R}\wedge S_{\bar{\delta}}\wedge\zeta^-}|^{\delta}] =  0. \end{eqnarray*}
\end{proof}

\begin{lemma}
Suppose  that Assumptions   \ref{assumption-Q-cont} and   \ref{Assum3} hold. Then % and \ref{Assum4}
 %  for $x \in \mathbb{R}^d$ and $k \in \mathbb{S}$ 
\begin{eqnarray}\label{limP[zeta<t]=0}
		\lim\limits_{|\tilde{x}-x|\to  0}\mathbb{P}\{\zeta \leq t\} = 0 
\end{eqnarray}
holds for every $t\geq 0 $.
\end{lemma}

\begin{proof}
Given $\epsilon > 0$. Choose $R$ sufficiently large so that $\mathbb{P}\{\tau_R \leq t\} < \epsilon$. Observe that
\begin{align} \label{P[Zeta<t]}
	\mathbb{P}\{\zeta \leq t\} &= \mathbb{P}\{\zeta \leq t, \tau_{R}< t\} + \mathbb{P}\{\zeta \leq t, \tau_{R} \geq t\} \nonumber\\
	&\leq  \mathbb{P}\{\tau_{R}< t\} + \mathbb{P}\{\zeta \leq t, \tau_{R} \geq t, S_{\bar{\delta}} \leq t\wedge\zeta\} + \mathbb{P}\{\zeta \leq t, \tau_{R} \geq t, S_{\bar{\delta}} > t\wedge\zeta\} \nonumber\\
	&\leq  \epsilon +\mathbb{P}\{\zeta \leq t, \tau_{R} \geq t, S_{\bar{\delta}} \leq t\wedge\zeta\} + \mathbb{P}\{\zeta \leq t, \tau_{R} \geq t, S_{\bar{\delta}} > t\wedge\zeta\} \nonumber
	\\
	&\leq  \epsilon + \mathbb{P}\{S_{\bar{\delta}} \leq t\wedge\zeta\wedge
	\tau_{R}\} + \mathbb{P}\{\zeta \leq t\wedge \tau_{R} \wedge S_{\bar{\delta}}\}.
\end{align}

As in the proof of Theorem \ref{thm-Feller}, condition \eqref{qklH} % \eqref{Ass sum|q-q| < K|x-y|}
 enables us to derive 
\begin{displaymath}
\P \{\zeta \leq t \wedge \tau_{R}\wedge S_{\bar{\delta}} \} \le \kappa_{R}\int_{0}^{t} \gamma_{k}(\mathbb{E}[|\Delta_{s\wedge\tau_{R}\wedge S_{\bar{\delta}}\wedge \zeta^{-}}| ] )ds.
\end{displaymath}
%  Now we compute
% \begin{align*}
%	\mathbb{P}& \{\zeta \leq t \wedge \tau_{R}\wedge S_{\bar{\delta}} \}\\
%	&= \mathbb{E}[1_{\{\tilde{\varLambda}(t \wedge \tau_{R}\wedge S_{\bar{\delta}}\wedge \zeta)- \varLambda(t \wedge \tau_{R}\wedge S_{\bar{\delta}}\wedge \zeta) \neq 0\}}]\\
%	&=\mathbb{E}[\int_{0}^{t \wedge \tau_{R}\wedge S_{\bar{\delta}}\wedge\zeta}\int_{\mathbb{R}^+}( 1_{\{\tilde{\varLambda}(s^{-})- \varLambda(s ^{-}) + h(\tilde{X}(s^-),\varLambda(s^-),z) - h(X(s^-),\varLambda(s^-),z)\neq 0\}}\\
%	& \qquad\qquad\qquad\qquad \qquad\qquad- 1_{\{\tilde{\varLambda}((s^{-})- \varLambda((s^{-}) \neq 0\}})m(dz)ds]\\
%	&= \mathbb{E}\bigg[\int_{0}^{t \wedge \tau_{R}\wedge S_{\bar{\delta}}\wedge\zeta}\int_{\mathbb{R}^+}1_{\{ h(\tilde{X}(s^-),\varLambda(s^-),z) - h(X(s^-),\varLambda(s^-),z)\neq 0\}}m(dz)ds\bigg]\\
%	&\leq \mathbb{E}\bigg[\int_{0}^{ \wedge \tau_{R}\wedge S_{\bar{\delta}}t\wedge\zeta}\sum_{l\in\mathbb{S}, l\neq \varLambda(s^-)}|q_{\varLambda(s^-),l}(\tilde{X}(s^-)) - q_{\varLambda(s^-),l}(X(s^-))|ds\bigg]\\
%	&\leq K\mathbb{E}\bigg[\int_{0}^{t\wedge \tau_{R}\wedge S_{\bar{\delta}}\wedge\zeta }|\tilde{X}(s^-)-X(s^-)|^{\delta}ds\bigg]\\
%	&\leq K\int_{0}^{t} \mathbb{E}\left[|\tilde{X}(s \wedge \tau_{R}\wedge S_{\bar{\delta}}\wedge\zeta^{-})-X(s \wedge \tau_{R}\wedge S_{\bar{\delta}}\wedge\zeta^{-})|^{\delta}\right] ds \\
%	&= K\int_{0}^{t} \mathbb{E}\left[|\Delta_{s\wedge\tau_{R}\wedge S_{\bar{\delta}}\wedge \zeta^{-}}|^{\delta}\right] ds
% \end{align*}
%  where the second inequality follows from (\ref{Ass sum|q-q| < K|x-y|}). It follows from
Furthermore,  \eqref{e1:limE[Del^d =0} implies  that
\begin{align}\label{limP1=0}
\lim\limits_{|\tilde{x}-x|\to  0}\mathbb{P}\{\zeta \leq t \wedge \tau_{R}\wedge S_{\bar{\delta}} \} = 0.
\end{align}
 On the set $\{S_{\bar{\delta}} \leq t\wedge \zeta\wedge \tau_{R}\}$ we have $\bar{\delta} \leq |\Delta_{S_{\bar{\delta}}\wedge t \wedge\zeta\wedge \tau_{R}}|$.
Since $G$ is increasing, we have
$$0 < G(\bar{\delta}) \leq G(|\Delta_{t\wedge S_{\bar{\delta}}\wedge\zeta\wedge \tau_{R}}|).$$
Thus
\begin{align*}
G(\bar{\delta})\mathbb{P}\{S_{\bar{\delta}} \leq t\wedge \zeta\wedge \tau_{R}\} \leq \mathbb{E}[G(|\Delta_{t\wedge S_{\bar{\delta}}\wedge\zeta\wedge \tau_{R}}|)1_{\{S_{\bar{\delta}} \leq t\wedge \zeta\wedge \tau_{R}\}}]
\leq \mathbb{E}[G(|\Delta_{t\wedge S_{\bar{\delta}}\wedge\zeta\wedge \tau_{R}}|)].
\end{align*}
This, together with   (\ref{limE[S(D)]=0}),  implies that 
\begin{align}\label{limP2=0}
	\lim\limits_{|\tilde{x}-x|\to  0}\mathbb{P}\{S_{\bar{\delta}} \leq t\wedge \zeta\wedge \tau_{R}\} = 0.
\end{align}
In view of (\ref{P[Zeta<t]}), it follows from (\ref{limP1=0}) and (\ref{limP2=0}) that $\lim_{|\tilde{x}-x|\to  0}\mathbb{P}\{\zeta \leq t\} \leq \epsilon.$ Since $\epsilon$ is arbitrary, we obtain \eqref{limP[zeta<t]=0}.
% \begin{eqnarray*} \lim\limits_{|\tilde{x}-x|\to  0}\mathbb{P}\{\zeta \leq t\} = 0.\end{eqnarray*} 
\end{proof}

\begin{lemma}\label{lem-P(t<T)-0}
	Suppose   that Assumptions \ref{assumption-Q-cont} and   \ref{Assum3} hold.  Then 
	%  for $x \in \mathbb{R}^d$ and $k \in \mathbb{S}$ 
	\begin{eqnarray}\label{limP[t<T]=0}
	\lim\limits_{|\tilde{x}-x|\to  0}\mathbb{P}\{t < T\} = 0  
	\end{eqnarray}
	holds for every $t\geq 0 $.
\end{lemma}

\begin{proof}  %  Given  $x \in \mathbb{R}^d$ and $k \in \mathbb{S}$. 
We may assume without loss of generality that $\bar{\delta} \geq |x-\tilde{x}| >\frac{1}{n_{0}} >  0$ for some $n_{0} \in \N$. 
Let $\epsilon > 0$ and choose a sufficiently large $R$ so that $\mathbb{P}\{\tau_R \leq t\} < \epsilon$. 
For  each $n\geq n_0$, define $T_{n}$ and $T$ as in \eqref{eq-Tn-defn} and \eqref{eq-T-defn}, respectively. 
 % Let $\tilde{x} \in \mathbb{R}^d$ be such that $\bar{\delta} \geq |x-\tilde{x}| > 0$, where $\bar{\delta}:= \delta_{0}\wedge\alpha$. Let $(X(t),\varLambda(t),\tilde{X}(t),\tilde{\varLambda}(t))$ denote the coupling process corresponding to the coupling operator $\tilde{\mathcal{L}}$ with initial condition $(x,k,\tilde{x},k)$. 

 % Suppose that $|x-\tilde{x}| > \frac{1}{n_0}$ for some $n_0 \in \mathbb{N}$. For each $n\geq n_0$, define
% \begin{eqnarray} T_n := \inf\{t\geq 0: |X(t) - \tilde{X}(t)| < \frac{1}{n} \}. \end{eqnarray} Then $\lim\limits_{n\longrightarrow\infty}T_n = T$, where \begin{eqnarray} T :=\inf\{t\geq 0: X(t) = \tilde{X}(t)\}. \end{eqnarray}
We first observe that
\begin{align}\label{P[t<T]}
\mathbb{P}\{t<T\} &= \mathbb{P}\{t<T, \tau_R < t\} + \mathbb{P}\{t<T, \tau_R \geq t\} \nonumber\\
&\leq \mathbb{P}\{\tau_R <t\} + \mathbb{P}\{t<T, \tau_R \geq t, S_{\bar{\delta}}<t\} + \mathbb{P}\{t<T, \tau_R \geq t, S_{\bar{\delta}} \geq t\} \nonumber\\
&\leq \epsilon + \mathbb{P}\{S_{\bar{\delta}} \leq t\wedge T\wedge \tau_{R}\} + \mathbb{P}\{t \leq T\wedge\tau_{R}\wedge S_{\bar{\delta}}\} \nonumber\\
&= \epsilon + \mathbb{P}\{S_{\bar{\delta}} \leq t\wedge T\wedge \tau_{R}, S_{\bar{\delta}} \leq \zeta\} + \mathbb{P}\{S_{\bar{\delta}} \leq t\wedge T\wedge \tau_{R},  S_{\bar{\delta}} > \zeta\}  \nonumber\\
& \quad \ +  \mathbb{P}\{t \leq T\wedge\tau_{R}\wedge S_{\bar{\delta}}, t<\zeta\} + \mathbb{P}\{t \leq T\wedge\tau_{R}\wedge S_{\bar{\delta}}, t \geq \zeta\} \nonumber\\
&\leq \epsilon + \mathbb{P}\{S_{\bar{\delta}} \leq t\wedge T\wedge \tau_{R}\wedge\zeta\} +  \mathbb{P}\{\zeta < S_{\bar{\delta}}\wedge t\wedge T\wedge \tau_{R}\}   \nonumber\\
& \quad \ +    \mathbb{P}\{t \leq T\wedge\tau_{R}\wedge S_{\bar{\delta}}\wedge\zeta\} + \mathbb{P}\{\zeta \leq t\} \nonumber\\
&\leq \epsilon + \mathbb{P}\{S_{\bar{\delta}} \leq t\wedge T\wedge \tau_{R}\wedge\zeta\} +  \mathbb{P}\{\zeta \leq t\}   +  \mathbb{P}\{t \leq T\wedge\tau_{R}\wedge S_{\bar{\delta}}\wedge\zeta\} + \mathbb{P}\{\zeta \leq t\} \nonumber\\
% & \quad \ +  \mathbb{P}\{t \leq T\wedge\tau_{R}\wedge S_{\bar{\delta}}\wedge\zeta\} + \mathbb{P}\{\zeta \leq t\} \nonumber\\
&=\epsilon + \mathbb{P}\{S_{\bar{\delta}} \leq t\wedge T\wedge \tau_{R}\wedge\zeta\} + \mathbb{P}\{t \leq T\wedge\tau_{R}\wedge S_{\bar{\delta}}\wedge\zeta\} + 2\mathbb{P}\{\zeta \leq t\} \nonumber\\
&\leq \epsilon + \mathbb{P}\{S_{\bar{\delta}} \leq T\wedge \tau_{R}\wedge\zeta\} + \frac{\mathbb{E}[T\wedge\tau_{R}\wedge S_{\bar{\delta}}\wedge\zeta]}{t} + 2\mathbb{P}\{\zeta \leq t\}.
\end{align}

Note that on the set $\{S_{\bar{\delta}} \leq T_n\wedge \tau_{R}\wedge \zeta\}$ we have $\bar{\delta} \leq |\Delta_{S_{\bar{\delta}}\wedge T_n \wedge \tau_{R}\wedge\zeta}|$.
Since $G$ is increasing, 
$0 < G(\bar{\delta}) \leq G(|\Delta_{S_{\bar{\delta}}\wedge T_n \wedge \tau_{R}\wedge\zeta}|).$
Thus
\begin{align*}
	G(\bar{\delta})\mathbb{P}\{S_{\bar{\delta}} \leq T_n\wedge \tau_{R}\wedge \zeta\} &\leq \mathbb{E}[G(|\Delta_{S_{\bar{\delta}}\wedge T_n \wedge \tau_{R}\wedge\zeta}|)1_{\{S_{\bar{\delta}} \leq T_n\wedge \tau_{R}\wedge \zeta\}}]
	\leq \mathbb{E}[G(|\Delta_{S_{\bar{\delta}}\wedge T_n \wedge \tau_{R}\wedge\zeta}|)]\\
	&= G(|x-\tilde{x}|) + \mathbb{E}\bigg[\int_{0}^{S_{\bar{\delta}}\wedge T_n \wedge \tau_{R}\wedge\zeta}\wdh \LL_{k} G(|\Delta_s|)ds\bigg]\\
	&\leq G(|x-\tilde{x}|) -\beta_{R}\mathbb{E}[T_n \wedge \tau_{R}\wedge S_{\bar{\delta}}\wedge\zeta],
\end{align*}
where the last inequality follows from (\ref{LS<beta}). So
\begin{eqnarray*}
	G(\bar{\delta})\mathbb{P}\{S_{\bar{\delta}} \leq T_n\wedge \tau_{R}\wedge \zeta\} +  \beta_{R}\mathbb{E}[T_n \wedge \tau_{R}\wedge S_{\bar{\delta}}\wedge\zeta] \leq G(|x-\tilde{x}|).
\end{eqnarray*}
Passing to the limit as  $n\rightarrow\infty$, we obtain
\begin{eqnarray*}
	G(\bar{\delta})\mathbb{P}\{S_{\bar{\delta}} \leq T\wedge \tau_{R}\wedge \zeta\} +  \beta_{R}\mathbb{E}[T \wedge \tau_{R}\wedge S_{\bar{\delta}}\wedge\zeta] \leq G(|x-\tilde{x}|).
\end{eqnarray*}
Then, in view of (\ref{P[t<T]}), we have
\begin{align*}% \label{P[t<T] < S+2P}
\mathbb{P}\{t<T\} &\leq \epsilon + \mathbb{P}\{S_{\bar{\delta}} \leq T\wedge \tau_{R}\wedge\zeta\} + \frac{\mathbb{E}[T\wedge\tau_{R}\wedge S_{\bar{\delta}}\wedge\zeta]}{t} + 2\mathbb{P}\{\zeta \leq t\} \nonumber\\
&\leq \epsilon + \frac{G(|x-\tilde{x}|)}{G(\bar{\delta})} + \frac{G(|x-\tilde{x}|)}{t\beta} + 2\mathbb{P}\{\zeta \leq t\}.
\end{align*}
From (\ref{limP[zeta<t]=0})  and the fact that $\lim_{|\tilde{x}-x|\rightarrow 0}G(|x-\tilde{x}|) = 0$, we obtain
% \begin{eqnarray*}
	$ \lim_{|\tilde{x}-x|\rightarrow 0}\mathbb{P}\{t<T\} \leq \epsilon. $ % \end{eqnarray*}
Since $\epsilon$ was arbitrary, we obtain \eqref{limP[t<T]=0}.  % $\lim\limits_{|\tilde{x}-x|\longrightarrow 0}\mathbb{P}\{t<T\} = 0$.
\end{proof}
%%%%%%%%%%%%%%%%%%%%%%%%%%%%%%%%%%%%%%%%%%%%%%%%%%%

%\newpage

Now we are ready to present the proof of  Theorem \ref{thm-str-Feller}. 
\begin{proof}[Proof of Theorem \ref{thm-str-Feller}]  Given  $x \in \mathbb{R}^d$ and $k \in \mathbb{S}$. We want to show that for every 
 %  bounded Borel measurable function $f$ on
 $f\in \B_{b}(\mathbb{R}^d\times \ss)$,   the limit $(P_{t}f)(\tilde{x},\tilde{k})\rightarrow  (P_{t}f)(x,k)$ as $(\tilde{x},\tilde{k})\rightarrow (x,k)$ holds for all $t > 0$.  As in the proof of Theorem \ref{thm-Feller}, we only need to consider the case when $\tilde{k}=k$.  %  Since $\mathbb{S} = \{ 1, 2,...\}$ has a discrete topology, we may consider only 

For any given $\epsilon > 0$ we can choose a sufficiently large $R$ so that $\mathbb{P}\{\tau_R \leq t\} < \epsilon$. Let $\tilde{x} \in \mathbb{R}^d$ be such that $\bar{\delta} \geq |x-\tilde{x}| > 0$, where $\bar{\delta}:= \delta_{0}\wedge\alpha$. Denote the coupling process corresponding to the coupling operator $\wdh \A$ defined in \eqref{eq:Ahat-operator-defn} with initial condition $(x,k,\tilde{x},k)$ by $(X(t),\varLambda(t),\tilde{X}(t),\tilde{\varLambda}(t))$. Denote by
\begin{eqnarray}
\tilde{T} :=\inf\{t\geq 0: (X(t),\varLambda(t)) = (\tilde{X}(t),\tilde{\varLambda}(t))\}
\end{eqnarray}
the coupling time of $(X(t),\varLambda(t))$ and $(\tilde{X}(t),\tilde{\varLambda}(t))$. Recall the stopping time $T$ defined in \eqref{eq-T-defn}. We make the following observations: 
% \begin{flushleft}
(i) $T \leq \tilde{T}$,  and 
%\end{flushleft}    \begin{flushleft}
(ii) $T < \zeta \text{ implies } T = \tilde{T}.$
% \end{flushleft}
We then have
\begin{align*} % \hspace{-3.5em}
1_{\{t<\tilde{T}\}} &= 1_{\{t<T\}} + 1_{\{T\leq t <\tilde{T}\}} \\
& =  1_{\{t<T\}} + 1_{\{T\leq t <\tilde{T}, \zeta \leq t\}} + 1_{\{T\leq t <\tilde{T}, \zeta > t\}}      \\
&\leq 1_{\{t<T\}} + 1_{\{\zeta \leq t\}} + 1_{\{T\leq t, t < \zeta,t <\tilde{T}\}} \\
&  \leq 1_{\{t<T\}} + 1_{\{\zeta \leq t\}} + 1_{\{T < \zeta,t <\tilde{T}\}} \\
&\leq 1_{\{t<T\}} + 1_{\{\zeta \leq t\}} + 1_{\{T = \tilde{T}, t <\tilde{T}\}} \\
& \leq 1_{\{t<T\}} + 1_{\{\zeta \leq t\}} + 1_{\{t < T\}} \\
&= 2\cdot 1_{\{t<T\}} + 1_{\{\zeta \leq t\}}.
\end{align*}
Then it follows that  
\begin{align*}\label{|Pf-Pf|}
	|(P_{t}f)(\tilde{x},\tilde{k})- (P_{t}f)(x,k)| &= |\mathbb{E}[f(\tilde{X}(t),\tilde{\varLambda}(t))] - \mathbb{E}[f(X(t),\varLambda(t))]| \\
	&\leq \mathbb{E}[|f(\tilde{X}(t),\tilde{\varLambda}(t)) - f(X(t),\varLambda(t))|1_{\{t < \tilde{T}\}}]   \\ 
	 & \quad + \mathbb{E}[|f(\tilde{X}(t),\tilde{\varLambda}(t))-f(X(t),\varLambda(t))|1_{\{t \geq \tilde{T}\}}] \\
	&= \mathbb{E}[|f(\tilde{X}(t),\tilde{\varLambda}(t)) - f(X(t),\varLambda(t))|1_{\{t < \tilde{T}\}}]  \\
	&\leq 2||f||_{\infty}\mathbb{E}[1_{\{t < \tilde{T}\}}] \\
	&\leq 2||f||_{\infty}\mathbb{E}[2\cdot 1_{\{t<T\}} + 1_{\{\zeta \leq t\}}]\\
	&= 4||f||_{\infty}\mathbb{P}\{t < T\} +  2||f||_{\infty}\mathbb{P}\{\zeta \leq t\}.
\end{align*}
A combination of  (\ref{limP[zeta<t]=0})  and (\ref{limP[t<T]=0})  then gives 
\begin{eqnarray*}
\lim\limits_{|\tilde{x}-x|\rightarrow 0}|(P_{t}f)(\tilde{x},\tilde{k})- (P_{t}f)(x,k)| = 0.
\end{eqnarray*}
This establishes  the  strong Feller property and concludes the proof.\end{proof}

%  Moreover, we obtain the following lemma.

%%%%%%%%%%%%%%%%%%%%%%%%%%%%%%%%%%%%%%%%%%%%%%%%%%%%%%%%

\section{Irreducibility}\label{sect-irr}

%%%%%%%%%%%%%%%%%%%%%%%%%%%%%%%%%%%%%%%%%%%%%%%%%%%%%%%%%%%%%%%%
This section aims to establish   irreducibility  for  the process $(X,\Lambda)$. The general approach can be described as follows. We first show that  for any   $k \in  \mathbb{S}$, the process $X^{(k)}$ of \eqref{SDE X^k} is strong Feller and  irreducible. Then we use a result in \citet*{XiYZ-19} to write $P(t,(x,k),B\times\{l\})$ as a convergent series in terms of sub-transition probabilities of the killed processes $\tilde X^{(j)}, j\in \ss$ and the transition rates $q_{jl}(x)$. 

 Denote the transition probability of the process $X^{(k)}$   by
\begin{eqnarray*}
P^{(k)}(t,x,B) := \mathbb{P}\{ X^{(k)}(t) \in B | X^{(k)}(0) = x\}, \quad B \in \B(\R^{d}). 
\end{eqnarray*} The corresponding semigroup $P_{t}^{(k)}$ is said to be {\em irreducible} if $P^{(k)}(t,x,B) > 0$ for all nonempty open set $B \subset\mathbb{R}^d$.
We next kill the process $X^{(k)}$ with killing rate $q_k(\cdot)$ and denote the killed process   by $\tilde{X}^{(k)}$, that is, we define 
\begin{eqnarray*}
\tilde{X}^{(k)}(t) =
\begin{cases}
X^{(k)}(t) & \text{ if }~~ t < \tau,  \\
\partial & \text{ if } ~~ t \geq \tau, 
\end{cases}
\end{eqnarray*}
where $\tau := \inf\{t\geq 0 : \varLambda(t) \neq \varLambda(0)\}$ and $\partial$ is a cemetery point added to $\mathbb{R}^d$. Then the semigroup of the killed process $\tilde{X}^{(k)}$ is given by
\begin{displaymath}
\tilde{P}_{t}^{(k)}f(x) :=\E_{x}[f(\tilde X^{(k)}(t))] =\mathbb{E}\bigg[f(X^{(k)}(t))\exp\bigg\{\int_{0}^{t}q_{kk}(X^{(k)}(s)) ds\bigg\}\bigg| X^{(k)}(0)=x\bigg],
\end{displaymath} where $ f \in \mathfrak{B}_{b}(\mathbb{R}^d)$.
 We also denote its sub-transition probability by
 \begin{eqnarray*}
 \tilde{P}^{(k)}(t,x,B) := \tilde{P}_{t}^{(k)}1_{B}(x)=  \E_{x} [1_{B}(\tilde{X}^{(k)}(t))]=\mathbb{P}\{ \tilde{X}^{(k)}(t) \in B | \tilde{X}^{(k)}(0) = x\},  \quad B \in \B(\R^{d}).  \end{eqnarray*}

\begin{lemma}\label{X(k) Strong Feller}
Under Assumptions \ref{Assum3} and \ref{weak solution X^(k)},   the semigroup $P_{t}^{(k)}$ is strong Feller. 
\end{lemma}

\begin{proof} Let $(\wdt X^{(k)}, X^{(k)} )$ be the coupling process corresponding to $\wdh\LL_{k} $ of \eqref{eq-hatLk-defn} with initial condition $(\tilde x, x)$. Suppose without loss of generality that $0< |\tilde x-  x| < \delta_{0}$, where $\delta_{0}$ is the positive constant in Assumption \ref{Assum3}. Define  $T:= \inf\{t\ge0: \wdt X(t) = X(t)\}$. Using very similar calculations as those in the proof of Lemma \ref{lem-P(t<T)-0}, we can show that  $\lim_{|\tilde{x}-x|\rightarrow 0}\mathbb{P}\{t<T\} = 0.$   Then, for any $f \in {\B}_{b}(\mathbb{R}^d)$ and $t > 0$, we have
	\begin{align*}
	|(P^{(k)}_{t}f)(\tilde{x})- (P^{(k)}_{t}f)(x)| &= |\mathbb{E}[f(\tilde{X}^{(k)}(t))] - \mathbb{E}[f(X^{(k)}(t))]| 
	\leq 2||f||_{\infty}\mathbb{P}\{t<T\} \to 0,
	\end{align*} as $\tilde x -x \to 0$. 
	This implies that $P^{(k)}_{t} f$ is a continuous function and hence  completes the proof.
\end{proof}
%%%%%%%%%%%%%%%%%%%%%%%%%%%%%%%%%%%%%%%%%%%%%%%%%%%%%%%%%%%%%%%%%%%%
% We need the following lemma:
\begin{lemma}
	Suppose  that  Assumption  \ref{Assump-linear growth} holds. Then for every $T>0$ there exists a constant $K := K(T,X(0)) > 0$ so that 
	\begin{equation}\label{E[X^2] < K}
	\mathbb{E}[|X(t)|^2] \leq K
\end{equation}
	for all $t \in [0,T]$.
\end{lemma}
\begin{proof} This lemma follows from  \eqref{< |x|^2+1} and standard arguments. For brevity, we omit the details here. \end{proof}

%%%%%%%%%%%%%%%%%%%%%%%%%%%%%%%%%%%%%%%%%%%%%%%%%%%%%%%%%%%%%%%%

%%To this end, 
To derive   irreducibility for  the semigroup $P_{t}^{(k)}$, we consider the function $F$ given by
\begin{eqnarray}
F(r) := \int_{0}^{\frac{r}{1+r}}e^{-\int_{0}^{s}g(w)dw}ds, \qquad r \in [0,\infty), 
\end{eqnarray}
where $g$ is the function given in Assumption \ref{Assum3}(ii). Since $g \geq 0$, we see that 
\begin{align}
0 &\leq F(r) \leq \frac{r}{1+r} \leq 1 \label{0<F<1},\\
0 &\leq F'(r) = \frac{1}{(1+r)^2}e^{-\int_{0}^{\frac{r}{1+r}}g(w)dw} \leq \frac{1}{(1+r)^2} \leq 1 \label{F'>0}, \intertext{and}
0 &\geq F''(r) = -\bigg[\frac{2}{1+r} + \frac{g(\frac{r}{1+r})}{(1+r)^2}\bigg]F'(r). \label{F''<0} 
\end{align}
In addition, for any $x \in \mathbb{R}^d$, we have  
\begin{align*}
\nabla F(|x|^2) =   2F'(|x|^2)x, \ \ \ \ 
\nabla^{2}F(|x|^2) = 4F''(|x|^2)xx^T + 2F'(|x|^2)I.
\end{align*}

% Follow the arguments used in \citet*{XiZ-19} and \citet*{Qiao-14} we obtain the following result.
\begin{lemma}\label{P^K irredicible}
Under Assumptions \ref{Assum1}, \ref{Assum3} (ii),  and \ref{Assump-linear growth}, the semigroup $P_{t}^{(k)}$ is irreducible.
\end{lemma}

\begin{rem}\label{rem-assume-3.1}
While irreducibility for jump diffusions has been considered in the literature such as  \citet*{Qiao-14,XiZ-19},  it is worth pointing out that Assumption   \ref{Assum3}(ii) is   much weaker than Assumptions (${\mathrm H_{1}'}$) and (${\mathrm H_{f}'}$) of \citet*{Qiao-14} and Assumption 2.5 of \citet*{XiZ-19}. In particular, as we mentioned in Remark \ref{rem-str-Feller},  Assumption   \ref{Assum3}(ii) allows to treat SDEs with merely H\"older continuous coefficients. The relaxations make the analyses more involved and subtle than those in the literature. Nevertherless, to preserve the flow of reading, we defer the proof of Lemma \ref{P^K irredicible} to Appendix \ref{sect-appendix}.
 % those in the aforementioned papers. % See Example \ref{ex1} for more details.}  % [To expand later!]  
\end{rem}

%%%%%%%%%%%%%%%%%%%%%%%%%%%%%%%%%%%%%%%%%%%%%%%%%%%%%%%%%%%%%%%%

\begin{proof}[Proof of Theorem \ref{thm-irreducibilty}] Given $t > 0$ and $(x,k) \in \mathbb{R}^d \times \mathbb{S}$. We want to show that $P(t,(x,k),B\times\{l\}) > 0$ for all $l \in \mathbb{S}$ and all $B \in \B(\mathbb{R}^d)$ with positive Lebesgue measure.   Under Assumption \ref{Assump-Q irreducible} and from Lemma \ref{X(k) Strong Feller}, as in the proof of Theorem 4.8 of \citet*{XiYZ-19},  we can write
\begin{equation}
\label{eq-P(t-xk-series}
\begin{aligned}
 &P (t,(x,k),B\times\{l\})\\
&\ = \delta_{kl}\tilde P^{(k)}(t,x,B) + \sum_{m=1}^{\infty}~~\idotsint\limits_{0<t_1<\cdots<t_m<t}
  \sum_{\substack{l_0,l_1,l_2,...,l_m \in\mathbb{S}\\l_i\neq l_{i+1},l_0=k,l_m=l}}\int\limits_{\mathbb{R}^d}\cdots\int\limits_{\mathbb{R}^d}\tilde{P}^{(l_0)}(t_1,x,dy_1)q_{l_0l_1}(y_1)\\
& \qquad \times \tilde{P}^{(l_1)}(t_2-t_1,y_1,dy_2)\cdots q_{l_{m-1}l_m}(y_m)\tilde{P}^{(l_m)}(t-t_m,y_m,B)dt_1dt_2\cdots dt_m,
\end{aligned}
\end{equation}
where $\delta_{kl}$ is the Kronecker symbol. From Assumption \ref{Assump-Q irreducible} (ii), we know that the set $\{y\in \mathbb{R}^d : q_{l_il_{i+1}}(y) > 0\}$ has positive Lebesgue measure. Then it suffices to show that $\tilde{P}^{(k)}(s,y,B) > 0$ for all $k\in\mathbb{S}$, $s>0$ and $B \in \B(\mathbb{R}^d)$. We calculate
\begin{align*}
\tilde{P}^{(k)}(s,y,B) &= \mathbb{P}\{\tilde{X}^{(k)}_{y}(s) \in B\}
 = \mathbb{E}_{k}\bigg[1_{B}(X^{(k)}_{y}(s))\exp\bigg(-\int_{0}^{s}q_{k}(X^{(k)}_{y}(r))dr\bigg) \bigg]\\
&\geq \mathbb{E}_{k}\left[1_{B}(X^{(k)}_{y}(s))e^{-M}\right]
 \geq  e^{-M}\mathbb{P}\{X^{(k)}_{y}(s) \in B\}
= e^{-M}P^{(k)}(s,y,B).
\end{align*}
From Lemma \ref{P^K irredicible}, the semigroup associated with the process $X^{(k)}$ is irreducible and therefore $P^{(k)}(s,y,B) > 0$. This completes the proof.\end{proof}

% \section{The Existence and Uniqueness of Invariant Measures} 
 % {\color{red}The following results can be proved in the same manner as in  \citet*{Xi-04} or \citet*{XiZ-19}.}
\begin{prop}\label{existence of invariant measure}
% Assume Assumptions \ref{Ass 2.1} and \ref{Ass 2.2} hold.
 Suppose  that  Assumptions \ref{Assum1},    \ref{assumption-non-lip}, and \ref{assumption-Q-cont} hold. In addition, assume there exist constants $\alpha, \beta > 0$, a compact subset $C \subset \mathbb{R}^d$,  a compact subset $N \subset \mathbb{S}$, a measurable function $f:\mathbb{R}^d\times\mathbb{S}\rightarrow [1,\infty)$, and a twice continuously differentiable function $V:\mathbb{R}^d\times\mathbb{S}\rightarrow [0,\infty)$ such that
\begin{equation} \label{AV(x,k) <-alpha V(x,k) + betaI}
	  \mathscr AV(x,k) \leq -\alpha f(x,k) + \beta 1_{C\times N}(x,k), \quad \forall (x,k) \in \R^{d} \times \ss.
\end{equation}
 Then the the semigroup $P_{t}$ of \eqref{eq:swjd-semigroup}   has an invariant probability measure $\pi$.
\end{prop}
\begin{proof}  Since the proof is very similar to those in   \citet*{Xi-04} or \citet*{XiZ-19}, we shall only give the sketch here. We first use \eqref{AV(x,k) <-alpha V(x,k) + betaI} and It\^o's formula to derive  
   %We follow the same argument used in the proof of Theorem 3.3 of \citet*{Xi-04}. From (\ref{AV(x,k) <-alpha V(x,k) + betaI}), we have
% 	\begin{eqnarray*}
%		0 &\leq&\mathbb{E}_{(x,k)}\left[V(X(t\wedge \tau_{R}),\vLa(t\wedge \tau_{R}))\right]\\
%		&=& V(x,k) + \mathbb{E}_{(x,k)}\left[\int_{0}^{t\wedge\tau_{R}}\mathscr AV(X(s),\vLa(s))ds\right] \\
%		&\leq& V(x,k) + \mathbb{E}_{(x,k)}\left[\int_{0}^{t\wedge\tau_{R}}\left(-\alpha f(X(s),\vLa(s)) + \beta I_{C\times N}(X(s),\vLa(s))\right)ds\right] \\
% 		&=& V(x,k) -\alpha \mathbb{E}_{(x,k)}\left[\int_{0}^{t\wedge\tau_{R}} f(X(s),\vLa(s))ds\right] + \beta\mathbb{E}_{(x,k)}\left[\int_{0}^{t\wedge\tau_{R}} I_{C\times N}(X(s),\vLa(s))ds\right]. 
%	\end{eqnarray*}
%	Since $f \geq 1$ we see that
%	\begin{eqnarray*}
%		\alpha \mathbb{E}_{(x,k)}\left[t\wedge\tau_{R}\right] &\leq& \alpha \mathbb{E}_{(x,k)}\left[\int_{0}^{t\wedge\tau_{R}} f(X(s),\vLa(s))ds\right]\\
%		&\leq& V(x,k) + \beta\int_{0}^{t\wedge\tau_{R}}\mathbb{E}_{(x,k)}\left[I_{C\times N}(X(s),\vLa(s))\right]ds\\
%		&\leq&  V(x,k) + \beta\int_{0}^{t}P(s,(x,k),C\times N)ds.
%	\end{eqnarray*}
%	Letting $R\longrightarrow \infty$ we obtain
	\begin{align*}
		\alpha t \leq V(x,k) + \beta\int_{0}^{t}P(s,(x,k),C\times N)ds, \qquad \forall t > 0,
	\end{align*}
% 	and hence
%	\begin{eqnarray*}
%		\frac{\alpha }{\beta} &\leq&  \frac{V(x,k)}{\beta t} + \frac{1}{t}\int_{0}^{t}P(s,(x,k),C\times N)ds.
%	\end{eqnarray*}
which, in turn,   implies that
	\begin{align}\label{alpha/beta < liminf 1/t P(s,(x,k),CN)}
  \liminf\limits_{t\longrightarrow\infty}\frac{1}{t}\int_{0}^{t}P(s,(x,k),C\times N)ds \ge 	\frac{\alpha }{\beta} > 0. 
	\end{align}
We have shown in Theorem \ref{thm-Feller}  that   the process $(X,\vLa)$ is Feller under Assumptions  \ref{Assum1}--\ref{assumption-Q-cont}. Then, in view of    \citet*{Foguel-69} and \citet*{Stettner-86} (see also the proof of Theorem 4.5 of \citet*{MeynT-93III}), \eqref{alpha/beta < liminf 1/t P(s,(x,k),CN)} implies that an invariant measure $\pi$ exists. 
 % 	As in the proof of Theorem 4.5 of \citet*{MeynT-93III}, {\color{red}the authors state} the following result from \citet*{Foguel-69} and \citet*{Stettner}. For any Feller process, there are two mutually exclusive possibilities: either an invariant probability measure exists, or
%	\begin{eqnarray}\label{Result from Fogurl and Stettner}
%	\lim\limits_{t\longrightarrow\infty} \sup_{\mu}\frac{1}{t} \int_{0}^{t}\int P(s,(x,k),C\times N)\mu(dx,dk)ds = 0 
%	\end{eqnarray}
%	for any compact set $C\times N \subset \mathbb{R}^d\times\mathbb{S}$, where the supremum is taken over all initial distributions $\mu$ on the state space $\mathbb{R}^d\times\mathbb{S}$. From (\ref{alpha/beta < liminf 1/t P(s,(x,k),CN)}), we know that (\ref{Result from Fogurl and Stettner}) is impossible. Then an invariant probability measure $\pi$ exists.
\end{proof}

\begin{prop}\label{prop-existence-uniqueness-invariant measure}
Suppose   that Assumptions \ref{Assum1},     \ref{assumption-Q-cont}, \ref{Assum3}, \ref{weak solution X^(k)}, \ref{Assump-linear growth},  and  \ref{Assump-Q irreducible} %, and  \ref{Assump openset irreducible} 
hold. If there exists  a twice continuously differentiable function $V:\mathbb{R}^d\times\mathbb{S}\rightarrow [0,\infty)$ such that \eqref{AV(x,k) <-alpha V(x,k) + betaI} holds, then the the semigroup $P_{t}$ of \eqref{eq:swjd-semigroup} has a unique invariant measure.
\end{prop}
\begin{proof}
The existence of an invariant measure follows directly from Proposition \ref{existence of invariant measure}.  For the uniqueness, we note that $P_{t}$ is strong Feller   and irreducible by Theorems \ref{thm-str-Feller} and \ref{thm-irreducibilty}, respectively. Then by   \citet*{Cerrai-01} and also \citet*{Hairer-16},   $P_t$   can admit at most one invariant measure. This completes the proof.
\end{proof}

\section{Examples}\label{sect-exms}
\begin{example}\label{ex1}
Consider the following SDE
\begin{equation} dX(t) = b(X(t),\Lambda(t))dt + \sigma(X(t),\varLambda(t))dW(t) + \int_{U}c(X(t^{-}),\varLambda(t^{-}),u)\tilde{N}(dt,du), 
\label{Eq Expo Example}
\end{equation}	
 with initial condition $X(0)= x \in \mathbb{R}$, where $W$ is a standard $1$-dimensional Brownian motion, $\tilde{N}$ is the associated compensated Poisson random measure on $[0,\infty)\times U$ with intensity $dt\nu(du)$ in which $U = \{u\in \mathbb{R} : 0 < |u| < 1\}$ and $\nu(du):= \frac{du}{|u|^{2}}$. Note that $\nu$ is a $\sigma$-finite measure on $U$ with $\nu(U) =\infty$. The component $\varLambda$ is the continuous-time stochastic process taking values in $\mathbb{S}=\{1,2,\dots\}$ generated by $Q(x) = (q_{kl}(x))$ where   $$q_{kl}(x) =\begin{cases}
    \frac{k}{3^{l+k}}\frac{1}{(1+l|x|^2)}   & \text{if }  k\neq l\\
     -\sum_{l\neq k}q_{kl}(x) & \text{otherwise}.
\end{cases}$$  %for $x\in \mathbb{R}$ and $ \in \mathbb{S}$ and $q_{k}(x) = -q_{kk}(x) = \sum_{l\neq k}q_{kl}(x)$. 
Furthermore, suppose the coefficients of \eqref{Eq Expo Example} are given by
\begin{align*}
		\sigma(x,k) = x^{\frac23} + 1,\quad 
		b(x,k) = -\frac{x}{2k^2}, \quad 
			c(x,k,u)  =  \frac{ux }{\sqrt 2k},  \quad\text{ for } (x,k) \in \R\times \ss \text{ and } u \in U.
	\end{align*} % for $(x,k) \in \R\times \ss$.

We make the following observations. \begin{enumerate}
  \item[(i)] Assumption \ref{Assum1} is satisfied. Indeed, one can verify directly that  the coefficients of \eqref{Eq Expo Example} satisfy the linear growth condition and Assumption 2.2 of \citet*{XiYZ-19}.  By Theorem 2.5 of \citet*{XiYZ-19}, \eqref{Eq Expo Example} has a unique strong  non-explosive solution.  This, of course, implies Assumption \ref{Assum1}. In addition, Assumption \ref{Assump-linear growth} holds. 
  {Indeed, for any $k \in \mathbb{S}$ and $x\in \mathbb{R}$, we have $2 \langle  x, b(x,k)\rangle = -\frac{x^{2}}{2k^{2}} $ and 
   	\begin{align*}
   		  |\sigma(x,k)|^2 + \int_{U}|c(x,k,u)|^2\nu(du)
   		& =   ( x^{\frac23} + 1)^2 + \frac{|x|^2}{2k^2} \int_{U}|u|^2\nu(du)\\
   		% &=  -\frac{|x|^2}{k^2} + ( x^{\frac23} + 1)^2 + \frac{|x|^2}{k^2}(\frac{1}{\sqrt{2}})^2\int_{(-1,1)}1du\\
   		&=    ( x^{\frac23} + 1)^2 + \frac{|x|^2}{k^2}\\
   	%	&= ( x^{\frac23} + 1)^2\\
   		&\leq 4[|x|^2 + 1].
   	\end{align*}
   Hence \eqref{< |x|^2+1} holds with $\kappa = 4$.  Since $a(x,k) = \sigma^{2}(x,k) =  ( x^{\frac23} + 1)^2 \ge 1$,   \eqref{eq1:elliptic} holds with  $\lambda=1$. }
 
   \item[(ii)] It is clear that Assumption \ref{assumption-non-lip} (i) holds true. Next we verify  Assumption \ref{assumption-Q-cont}. To this end, we compute
   % can assume without loss of generality  that $|x| \leq |y|$. Then	
\begin{align*}
	\sum_{l\in \mathbb{S}\backslash\{k\}}|q_{kl}(x)-q_{kl}(y)| &= \sum_{l\in \mathbb{S}\backslash\{k\}}\bigg|\frac{k}{3^{l+k}}\frac{1}{(1+l|x|^2)}-\frac{k}{3^{l+k}}\frac{1}{(1+l|y|^2)}\bigg|\\
	&= \frac{k}{3^k}\sum_{l\in \mathbb{S}\backslash\{k\}}\frac{1}{3^l}\bigg|\frac{1}{1+l|x|^2}-\frac{1}{1+l|y|^2}\bigg|  \\
	&\leq   \sum_{l\in \mathbb{S}}\frac{l}{3^l}\frac{||y|^2 - |x|^2|}{(1+l|x|^2)(1+l|y|^2)} \\
	&=  \sum_{l\in \mathbb{S}}\frac{l}{3^l}\frac{(|y|+|x|)||y|-|x|| }{(1+l|x|^2)(1+l|y|^2)}  \\
	% &\leq  \sum_{l\in \mathbb{S}}\frac{l}{3^l}\frac{2|y|}{(1+l|x|^2)(1+l|y|^2)}||y|-|x||  \\ 
	& \le  \sum_{l\in \mathbb{S}}\frac{l}{3^l} |y-x| 
	 = \frac34 |x-y|, 
	% &\leq 2|y-x|\sum_{l\in \mathbb{S}}\frac{l}{2^l}\\
	% &= 4k|x-y|.
\end{align*} where the   last inequality follows from  the triangle inequality $||y|-|x|| \le |x-y|$ and the observation that \begin{align*} 
 \frac{|y|+|x|} {(1+l|x|^2)(1+l|y|^2)} & \le \frac{|y|}{ 1+l|y|^2} + \frac{|x|}{ 1+l|x|^2} \le \frac{|y|}{ 1+|y|^2} + \frac{|x|}{ 1+|x|^2} \le \frac12+\frac12 =1.   
\end{align*}
%  Consequently Theorem \ref{thm-soln-existence-uniqueness} says that \eqref{Eq Expo Example}  is Feller continuous. 
  \item[(iii)]  We can further verify that Assumption \ref{Assum3} holds. Indeed,  since $a(x,k) = \sigma^{2}(x,k)  = x^{\frac43} + 2x^{\frac23} + 1$,  for each $R > 0$, we can take $\lambda_{R} =1$ and $\sigma_{\lambda_{R}}(x,k) = (x^{\frac43} + 2x^{\frac23})^{\frac12}$ for all $(x,k) \in \R\times \ss$. Then it is straightforward to verify that  for all $x,z\in \R$ with $|x| \vee |z| \le R$ and $k\in \ss$ \begin{align*}
	&   |\sigma_{\lambda_{R}}(x,k) - \sigma_{\lambda_{R}}(z,k)|^2 + 2\langle x-z,b(x,k)-b(z,k)\rangle + \int_{U}|c(x,k,u)-c(z,k,u)|^2\nu(du) \\ 
	  &\  \le 2 (z^{\frac23} + x^{\frac23} + 2)(z^{\frac23} - x^{\frac23})-\frac{1}{k^2}|x-z|^2 +\frac{1}{k^2}|x-z|^2\\
		&\ \leq 4 ( R^{ \frac23}+1)|x-z|^{\frac23}  \\
		&\ = 4 ( R^{ \frac23}+1) |x-z| g( |x-z|),
		% &=&  \kappa_{R}|x-z|g(|x-z|)
\end{align*} where $g(r) = r^{-\frac13}$. Note that the function $g$ satisfies \eqref{g-integrable-01}. As a result, \eqref{Eq Expo Example}  is  strong Feller continuous  by Theorem \ref{thm-str-Feller}.  
\item[(iv)] Next we see immediately  that Assumptions \ref{Assump-Q irreducible}   holds and hence \eqref{Eq Expo Example} is irreducible by virtue of Theorem \ref{thm-irreducibilty}.  
 %  Indeed it is straightforward to see that Assumption  \ref{Assump-Q irreducible} holds. To verify Assumption \ref{Assump openset irreducible}, we notice that  
%  \item  This example shows that our result extends the result in \citet*{XiYZ-19}. Indeed, 
\end{enumerate} 
\end{example}

\begin{example}\label{ex2}
	Consider the following SDE
	\begin{equation}\label{Eq2 Expo Example} 
	\begin{aligned}
	dX(t) & = b(X(t),\Lambda(t))dt + \sigma(X(t),\varLambda(t))dW(t) + \int_{U}c(X(t^{-}),\varLambda(t^{-}),u)\tilde{N}(dt,du), \\ X(0) & = x \in \mathbb{R}^2,
	\end{aligned}\end{equation}	
	where $W$ is a standard $2$-dimensional Brownian motion, $\tilde{N}$ is the associated compensated Poisson random measure on $[0,\infty)\times U$ with intensity $dt\nu(du)$ in which $U = \{u\in \mathbb{R}^2 : 0 < |u| < 1\}$ and $\nu(du):= \frac{du}{|u|^{2+\delta}}$ for some $\delta \in (0,2)$. The component $\varLambda$ is the continuous-time stochastic process taking values in $\mathbb{S}=\{1,2,\dots\}$ generated by $Q(x) = (q_{kl}(x))$ with
	$q_{kl}(x) =  \frac{2+\cos(k|x|)}{3^{l}(2+\sin(|x|^2))}$ for $x\in \mathbb{R}^2$ and $k\neq l \in \mathbb{S}$ and $q_{k}(x) = -q_{kk}(x) = \sum_{l\neq k}q_{kl}(x)$.  The coefficients of \eqref{Eq2 Expo Example} are given by
	\begin{align*}
		\sigma(x,k) = \frac{|x| + 1}{4} I,\quad 
		b(x,k) = -\frac{k}{k+1}x,\quad 
		c(x,k,u)  = \frac{\sqrt k}{\sqrt{k+1}} \gamma |u|x	\end{align*}
	where $I$ is the 2-dimensional identity matrix and $\gamma$ is a positive constant so that $\gamma^2\int_{U}|u|^2\nu(du) = 1$.
	
	Detailed calculations as those in Example \ref{ex1} reveal that \eqref{Eq2 Expo Example} has a unique non-explosive weak solution, which is strong Feller continuous and irreducible. Next we verify that   $V(x,k):=  |x|^2 +k$ satisfies (\ref{AV(x,k) <-alpha V(x,k) + betaI}) and hence by Proposition \ref{prop-existence-uniqueness-invariant measure},  \eqref{Eq2 Expo Example} has a unique invariant measure.
	
	Observe that $\nabla V(x,k) = 2x$ and $\nabla^{2}V(x,k) = 2I$.  We compute
		\begin{align*}
		\A V(x,k) %\mathcal{L}_{k}V(x,k) 
		&:= \frac{1}{2}\tr\left(a(x,k)\nabla^{2}V(x,k)\right) + \langle b(x,k),\nabla V(x,k)\rangle  + \sum_{l\in\mathbb{S}}q_{kl}(x)\left[V(x,l) - V(x,k)\right]\\
		& \quad \ + \int_{U}\left(V(x+c(x,k,u), k) - V(x,k) -\langle\nabla V(x,k), c(x,k,u)\rangle\right)\nu(du)\\
		&\le  \frac{1}{2}\tr\left(\frac{(|x|+1)^2}{16} 2I\right) -\frac{k}{k+1} \langle x,2x\rangle  + \sum_{l\neq k}  \frac{(2+\cos(k|x|)) (l-k)}{3^{l}(2+\sin(|x|^2))}  \\ & \quad\  + \frac{k}{k+1}\int_{U} \gamma^{2}  |x|^{2} |u|^{2}\nu(du)\\
		&= \frac{(|x|+1)^2}{8}   -\frac{2k}{k+1} |x|^2 + \frac{k}{k+1} |x|^{2} +  \frac{2+\cos(k|x|)}{2+\sin(|x|^2)}\sum_{l\neq k}\frac{l-k}{3^{l}}  \\
		& \le  \frac{|x|^{2} + 1}{4} -  \frac{k}{k+1} |x|^{2} +  \frac{2+\cos(k|x|)}{2+\sin(|x|^2)}\bigg( \frac34 - \frac k2\bigg)\\
		& \le -\frac14|x|^{2} + \frac{5}{2} - \frac{k}{6} \\
		& \le -\frac16 V(x,k) + \frac{5}{2},
 			% 	&= 2k(|x|+1)^2 - 4k^2|x|^2 + \frac{k}{3^k}\sum_{l\in\mathbb{S}}\frac{1}{3^{l}} \\
		% & \quad \ + \int_{U}\left(k(1+2\gamma\sqrt{k}|u|+\gamma^2k|u|^2)|x|^2 - k|x|^2 - 2\gamma k\sqrt{k}|u||x|^2\right)\nu(du)\\
		%&\le  2k(|x|+1)^2 - 4k^2|x|^2 + 1 \\
	%	& \quad \ + \int_{U}\left(k(1+2\gamma\sqrt{k}|u|+\gamma^2k|u|^2)|x|^2 - k|x|^2 - 2\gamma k\sqrt{k}|u||x|^2\right)\nu(du)\\
	%	&= 2k(|x|+1)^2 - 4k^2|x|^2  +1 + k^2|x|^2\gamma^2\int_{U}|u|^2\nu(du)\\
	%	&=  2k(|x|+1)^2 - 4k^2|x|^2  + 1 + k^2|x|^2\\
	%	&=  2k(|x|^2 + 2|x| + 1) - 3k^2|x|^2 + 1    \\
	%	&\leq  2k(|x|^2 + 2|x| + 1) - 3k|x|^2 + 1    \\
	%	&= 2k(2|x|+1) - k|x|^2  + 1 \\
	%	%&\leq 2k(|x|^2+2|x|+1) - 2k|x|^2  + 1\\
	%	%& = k[1+2|x|-2|x|^2] + 1\\
	%	& = \frac{k[4|x| + 2 - |x|^2] +1}{1+k|x|^2}V(x,k) .
		\end{align*} for all $(x,k) \in \R^{2}\times \ss$.  % where $\beta > 0$ is a sufficiently large positive constant. 
	%	Note that there exists some positive real number $r_0$ such that for all  $|x|>r_0$, we have $2+4|x| \leq \frac{1}{2}|x|^2$. Thus it follows that   
	%	\begin{align*}
	%	\mathscr A V(x,k)  \leq  \bigg( -\frac{k|x|^2}{2(1+k|x|^2)} + \frac{1}{1+k|x|^2}\bigg)V(x,k), \quad \forall (x,k) \in \{x\in \R^{d}: |x| \ge r_{0} \}\times \ss.
	%	\end{align*}
	%	Moreover, for all $k\in \ss$ and $|x| \ge r_{0}$, we have $\frac{k|x|^{2}}{1+k|x|^2} \ge  \frac{|x|^{2}}{1+|x|^2}  \ge  \frac{r_{0}^{2}}{1+ r_{0}^2} =:   4\alpha  > 0$. Also notice that there exists some $r_{1} > 0$ such that for all $|x| \ge r_{1}$ and $k \in \ss$, we have $\frac{1}{1+k|x|^2} \le \frac{1}{1+|x|^2} \le \alpha$. Then
	%	\begin{eqnarray*}
	%		\mathscr A V(x,k)	  \leq   -\alpha V(x,k) \quad  \forall |x| \geq r_0 \vee r_1, k \in \mathbb{S}.
	%	\end{eqnarray*} 
		% sufficiently large, say $|x|>r_1$ for some positive real number $r_1$, there exists $\alpha > 0$ so that $\frac{k|x|^{2}}{1+k|x|^2} > \alpha$ for all $k\in\mathbb{S}$. Hence 
	%	Consequently it follows that for some sufficiently large $\beta > 1$, we have  \begin{eqnarray*}
	%		\mathscr A V(x,k)	  \leq   -\alpha V(x,k) +  \beta ,  \quad \forall (x,k) \in \R^{d}\times \ss.
	%	\end{eqnarray*}
		Since $V(x,k) \rightarrow \infty$ as $|x|\vee k \rightarrow \infty$, this apparently implies  (\ref{AV(x,k) <-alpha V(x,k) + betaI}) and hence a unique invariant probability measure $\pi$  for  \eqref{Eq2 Expo Example} exists.
\end{example}		
%%%%%%%%%%%%%%%%%%%%%%%%%%%%%%%%%%%%%%%%%%%%%%%%%%%%%%%%

\noindent{\bf Acknowledgement.} The authors would like to thank the Associate Editor and the anonymous reviewers for their helpful comments and suggestions. 
%%%%%%%%%%%%%%%%%%%%%%%%%%%%%%%%%%%%%%%%%%%%%%%%%%%%%%%%%%
%%%%%%%%%%%%%%%%%%%%%%%%%%%%%%%%%%%%%%%%%%%%%%%%%%%%%%%%%%

%%%%%%%%%%%%%%%%%%%%%%%%%%%%%%%%%%%%%%%%%%%%%%%%%%%%%%%%
%%%%%%%%%%%%%%%%%%%%%%%%%%%%%%%%%%%%%%%%%%%%%%%%%%%%%%
%%%%%%%%%%%%%%%%%%%%%%%%%%%%%%%%%%%%%%%%%%%%%%%%%%%%%%%%%%%%%%%
\appendix\section{Proofs of Several Technical Results}\label{sect-appendix}
\begin{proof}[Proof of Lemma \ref{lem-24}]  
We will prove the lemma separately for the cases $d=1$ and $d \ge 2$.

{\em Case {\em (i)}: $d=1$.} Let $\{a_{n}\}$ be a strictly decreasing sequence of real numbers satisfying $a_{0}=1$,  $\lim_{n\to\infty}a_{n} =0$, and $\int_{a_{n}}^{a_{n-1}} \frac{\d r}{r} = n$ for each $n \ge 1$.
For each $n \ge 1$, let $\rho_{n}$ be a nonnegative continuous function with support on $(a_{n}, a_{n-1})$ so that
\begin{displaymath}
\int_{a_{n}}^{a_{n-1}}\rho_{n}(r) \d r =1 \text{ and } \rho_{n}(r) \le 2(nr)^{-1} \text{ for all }r > 0.
\end{displaymath}
For $x\in \R$, define \begin{equation}
\label{eq-fn psi-n}
\psi_{n}(x) = \int_{0}^{|x|} \int_{0}^{y}\rho_{n}(z) \d z\d y.
\end{equation} We can immediately verify that $\psi_{n}$ is even and twice continuously differentiable, with
\begin{equation}\label{eq sgn of psi'}
\psi_{n}'(r) =\sgn(r) \int_{0}^{|r|} \rho_{n}(z) \d z =\sgn(r) |\psi_{n}'(r)|,
\end{equation} and \begin{equation}
\label{eq-psi estimates}
 |\psi_{n}'(r)| \le 1,\quad  0 \le |r| \psi_{n}''(r) = |r| \rho_{n}(|r|) \le \frac2n,\quad  \text{and}\quad\lim_{n\to\infty} \psi_{n}(r) = |r|
\end{equation} for $r\in \R$. Furthermore, for each $r > 0$, the sequence $\{\psi_{n}(r) \}_{n\ge 1}$ is nondecreasing.  For each $n\in \N$, one can show that  $\psi_{n}$, $\psi_{n}'$, and $\psi_{n}''$ all vanish  on the interval $(-a_{n}, a_{n})$.
Moreover the classical arguments using Assumption \ref{assumption-non-lip} (i), \eqref{eq sgn of psi'}   and \eqref{eq-psi estimates}  reveal that
\begin{align}\label{eq-Lk psi-n estimate}
\nonumber \wdt \LL_{k} \psi_{n} (x-z) & =\frac12  \psi_{n}'' (x-z) |\sigma(x,k)-\sigma(z,k)|^{2} +\psi_{n}'(x-z) (b(x,k)- b(z,k))\\  &\ \quad  + \nonumber\int_{U}  [\psi_{n}(x-z + c(x,k,u)- c(z,k,u)) \\ \nonumber& \qquad \qquad \qquad 
- \psi_{n}(x-z)- \psi_{n}' (x-z) ( c(x,k,u)- c(z,k,u))]\nu(\d u)
\\ & \le   K \frac{\kappa_{R}}{n} +  \kappa_{R} \rho(|x-z|), \end{align}
   for all $x,z$   with  $|x| \vee |z| \le R$
 and  $0<  |x-z|  \le \delta_{0}$, where $K $ is a positive constant independent of $R$ and $n$.  Then it follows that 
  \begin{align*}
\E [  \psi_{n}(\Delta_{t\wedge S_{\delta_{0}} \wedge\tau_{R}\wedge\zeta})] & = \E[\psi_{n} (\tilde X(t\wedge S_{\delta_{0}} \wedge\tau_{R}\wedge\zeta)-X(t\wedge S_{\delta_{0}} \wedge\tau_{R}\wedge\zeta))]\\ 
 & = \psi_{n}(\tilde x-x) + \E\biggl[\int_{0}^{t\wedge \tau_{R} \wedge S_{\delta_{0}}\wedge\zeta}\wdt\LL_{k} \psi_{n}(\tilde X(s) - X(s))\d s\bigg]\\
    & \le  \psi_{n}(|\Delta _{0} |) + \E\biggl[\int_{0}^{t\wedge \tau_{R} \wedge S_{\delta_{0}}\wedge\zeta} \bigg( \kappa_{R} \rho (|\Delta_{s}|) + K \frac{\kappa_{R}}{n}\bigg) \d s\bigg]\\
    & \le  \psi_{n}(|\Delta _{0} |) +K \frac{\kappa_{R}}{n}t + \kappa_{R}  \int_{0}^{t} \rho\big(\E[|\Delta_{s\wedge \tau_{R} \wedge S_{\delta_{0}}\wedge\zeta }|]\big)\d s,
\end{align*} where the first inequality follows from \eqref{eq-Lk psi-n estimate} and the second inequality follows from the concavity of $\rho$ and Jensen's inequality. 
 %  we can use the same computations as those in the the proof of Lemma 2.4 of \citet*{XiYZ-19} to show that 
Then we use  the monotone convergence theorem and \eqref{eq-psi estimates}  to derive \begin{align*}
	\mathbb{E}[|\Delta_{t\wedge\tau_{R}\wedge S_{\delta_{0}}\wedge\zeta}|] \leq |\Delta_{0}| +  \kappa_{R}\int_{0}^{t}\rho(\mathbb{E}[|\Delta_{s\wedge\tau_{R}\wedge S_{\delta_{0}}\wedge\zeta}|])ds.
\end{align*}

Let $u(t):=\mathbb{E}[|\Delta_{t\wedge\tau_{R}\wedge S_{\delta_{0}}\wedge\zeta}|]$. Then $u$ satisfies
\begin{eqnarray*}
	0 \leq u(t) \leq v(t) := |\Delta_{0}| +  \kappa_{R}\int_{0}^{t}\rho(u(s))ds.
\end{eqnarray*}
Define the function $\Gamma(r):= \int_{1}^{r}\frac{ds}{\rho(s)}$ for $r>0$. Thanks to (\ref{rhoProperties1}), we can verify that $\Gamma$ is nondecreasing and satisfies $\Gamma(r) > -\infty$ for all $r>0$ and $\lim_{r\rightarrow0}\Gamma(r) = -\infty$. Then we have
\begin{align*}
	\Gamma(u(t)) &\leq \Gamma(v(t))
	= \Gamma(|\Delta_{0}|) + \int_{0}^{t}\Gamma'(v(s))v'(s)ds= \Gamma(|\Delta_{0}|) + \kappa_{R}\int_{0}^{t}\frac{\rho(u(s))}{\rho(v(s))}ds\\
	&% = \Gamma(|\Delta_{0}|) + \kappa_{R}\int_{0}^{t}\frac{\rho(u(s))}{\rho(v(s))}ds
	\leq  \Gamma(|\Delta_{0}|) + \kappa_{R}\int_{0}^{t}1ds
	=  \Gamma(|\Delta_{0}|) + \kappa_{R}t,
\end{align*}
where we use the assumption that $\rho$ is nondecreasing to obtain the last inequality. Taking the limit $|\Delta_{0}| = |\tilde{x}-x|\to  0$ we have $\Gamma(u(t)) \to  -\infty$ since $\lim_{r\rightarrow0}\Gamma(r) = -\infty$. Moreover, since $\Gamma(r) > -\infty$ for all $r>0$ we must have $\lim_{|\tilde{x}-x|\to  0}u(t)=0 $. This gives  \eqref{e1:EDd-=0} as desired. 
% \begin{eqnarray}\label{Eqd1} \lim\limits_{|\tilde{x}-x|\to  0}u(t) = \lim\limits_{|\tilde{x}-x|\to  0}\mathbb{E}[|\Delta_{t\wedge\tau_{R}\wedge S_{\delta_{0}}\wedge\zeta}|] =  0. \end{eqnarray}
% A slight modification of the above argument also shows that
%\begin{eqnarray}\label{eq-Eqd1}	\lim\limits_{|\tilde{x}-x|\to  0}\mathbb{E}[|\Delta_{t\wedge\tau_{R}\wedge S_{\delta_{0}}\wedge\zeta^{-}}|] =  0, \end{eqnarray} which, together with the  H\"{o}lder inequality,  implies  \begin{eqnarray*}\lim\limits_{|\tilde{x}-x|\to  0}\mathbb{E}[|\Delta_{t\wedge\tau_{R}\wedge S_{\delta_{0}}\wedge \zeta^{-}}|^{\delta}] = 0. \end{eqnarray*}
% Furthermore, since $0 \leq H(r) \leq (\frac{r}{2}\vee 1)$ for all $r \geq 0$, it follows from (\ref{Eqd1}) and \eqref{eq-Eqd1} that
% \begin{eqnarray*} \lim\limits_{|\tilde{x}-x|\to  0}\mathbb{E}[H(|\Delta_{t\wedge\tau_{R}\wedge S_{\delta_{0}}\wedge\zeta}|)] =\lim\limits_{|\tilde{x}-x|\to  0}\mathbb{E}[H(|\Delta_{t\wedge\tau_{R}\wedge S_{\delta_{0}}\wedge\zeta^{-}}|)] =  0.\end{eqnarray*}

{\em Case {\em (ii)} $d\ge 2$.}  Consider the function $f(x,z) : = |x-z|^{2}$. Then Assumption \ref{assumption-non-lip} (ii) implies that 
\begin{align*}
\wdt \LL_{k} f(x,z)&  = 2\langle x-z,  b(x,k)-b(z,k)\rangle   + |\sigma(x,k)-\sigma(z,k)|^2  + \int_{U}|c(x,k,u)-c(z,k,u)|^2\nu(du)\\ & \leq \kappa_{R} \rho(|x-z|^2),
\end{align*} for all $x, z \in \mathbb{R}^d$ with $|x|\vee |z| \leq R$ and $|x-z|\leq \delta_{0}$. Consequently 
\begin{align*} 
 \E\big[|\Delta_{t\wedge\tau_{R}\wedge S_{\delta_{0}}\wedge\zeta}|^{2}\big]  & = \E[f(\tilde X(t\wedge\tau_{R}\wedge S_{\delta_{0}}\wedge\zeta), X(t\wedge\tau_{R}\wedge S_{\delta_{0}}\wedge\zeta))]   \\
  & = f(\tilde x, x) + \E\bigg[\int_{0}^{t\wedge\tau_{R}\wedge S_{\delta_{0}}\wedge\zeta}\wdt\LL_{k} f(\tilde X(s), X(s))ds \bigg] \\
   & \le |\Delta_{0}| + \E\bigg[\int_{0}^{t\wedge\tau_{R}\wedge S_{\delta_{0}}\wedge\zeta}\kappa_{R} \rho(|\tilde X(s)- X(s)|^2)ds \bigg]\\
   & \le  |\Delta_{0}| + \kappa_{R} \int_{0}^{t} \rho(\E[|\Delta_{s\wedge\tau_{R}\wedge S_{\delta_{0}}\wedge\zeta}|^{2}])ds,
\end{align*}where the last inequality follows from   the concavity of $\rho$ and Jensen's inequality. Using the same argument as that in Case (i), we can show that $\lim_{|\tilde{x}-x|\to  0}\E[|\Delta_{t\wedge\tau_{R}\wedge S_{\delta_{0}}\wedge\zeta}|^{2}] =0$; which, together with H\"older's inequality, leads to \eqref{e1:EDd-=0}. Combining the two cases completes the proof.
\end{proof}

\begin{proof}[Proof of Lemma \ref{lem-LkG-estimate}] Let use first prove the lemma for the case when $d \ge 2$.  In view of (\ref{Omega_d}), it follows from (\ref{S'>0, S''<0}) that
\begin{align}\label{Sdrift}
&  \wdh{\varOmega}_{\mathrm d}^{(k)}G(|x-z|) \nonumber\\
 &\ \ = \frac{G''(|x-z|)}{2}\bar{A}(x,k,z,k) + \frac{G'(|x-z|)}{2|x-z|}[\tr A(x,k,z,k) - \bar{A}(x,k,z,k) +2B(x,k,z,k)] \nonumber\\
&\ \ \leq \frac{G''(|x-z|)}{2}4\lambda_{R} + \frac{G'(|x-z|)}{2|x-z|}[|\sigma_{\lambda_{R}}(x,k)-\sigma_{\lambda_{R}}(z,k)|^2 +2B(x,k,z,k)] \nonumber\\
&\ \ = 2\lambda_{R}\left(-1 - \frac{\kappa_{R}}{2\lambda_{R}}g(|x-z|)G'(|x-z|)\right) \nonumber\\ 
 & \ \ \qquad  + \frac{G'(|x-z|)}{2|x-z|}[|\sigma_{\lambda_{R}}(x,k)-\sigma_{\lambda_{R}}(z,k)|^2 +2B(x,k,z,k)] \nonumber\\
&\ \ = -2\lambda_{R} + \left(-\kappa_{R}g(|x-z|) + \frac{|\sigma_{\lambda_{R}}(x,k)-\sigma_{\lambda_{R}}(z,k)|^2 +2B(x,k,z,k)}{2|x-z|}\right)G'(|x-z|).
\end{align}
Since the function $G$ is concave, we have $G(r_1) - G(r_0) \leq G'(r_0)(r_1 - r_0)$ for all $r_0, r_1 \geq 0$. Take $r_0 = |x-z|$ and $r_1 = |x+c(x,k,u)- z-c(z,k,u)|$  % and note the fact that $|a+b|^2 = |a|^2 + |b|^2 + 2\langle a,b\rangle$ for all $a , b \in \mathbb{R}^d$
 to obtain \begin{align*}% \label{Sjump}
  & G(|x+c(x,k,u)- z-c(z,k,u)|) - G(|x-z|) -\frac{G'(|x-z|)}{|x-z|}\langle x-z, c(x,k,u) -c(z,k,u)\rangle \\
&\leq G'(|x-z|)\left(|x+c(x,k,u)- z-c(z,k,u)| - |x-z| - \frac{\langle x-z, c(x,k,u) -c(z,k,u)\rangle}{|x-z|}\right).
 %   \\ &\quad  -\frac{G'(|x-z|)}{|x-z|}\langle x-z, c(x,k,u) -c(z,k,u)\rangle 
 \end{align*} 
Furthermore, with $a: = x-z$ and $b: = c(x,k,u)-c(z,k,u)$, we can verify directly that  \begin{displaymath}
|a+b| -|a| -\frac{\langle a,b\rangle}{|a|} = \frac{-(|a+b| -|a|)^{2} + |b|^{2}}{2|a|} \le \frac{|b|^{2}}{2|a|}.
\end{displaymath}
 Hence it follows that 
\begin{align*}
G&(|x+c(x,k,u)-z-c(z,k,u)|) - G(|x-z|) -\frac{G'(|x-z|)}{|x-z|}\langle x-z, c(x,k,u) -c(z,k,u)\rangle \\
% &=  G'(|x-z|)\left(|x+c(x,k,u)- z-c(z,k,u)| - |x-z|\right)\nonumber\\ 
% & \quad-\frac{G'(|x-z|)}{|x-z|}\frac{|x-z + c(x,k,u)- c(z,k,u)|^2 - |x-z|^2 - |c(x,k,u)- c(z,k,u)|^2}{2} \nonumber\\
% &=  \frac{G'(|x-z|)}{2|x-z|}\left(2|x-z||x+c(x,k,u)- z-c(z,k,u)| - 2|x-z|^2\right)\nonumber\\ 
% & \quad -\frac{G'(|x-z|)}{2|x-z|}\left(|x-z + c(x,k,u)- c(z,k,u)|^2 - |x-z|^2 - |c(x,k,u)- c(z,k,u)|^2\right) \nonumber\\
% &= \left( \frac{-\left(|x-z| -|x-z + c(x,k,u)- c(z,k,u)|\right)^2 + |c(x,k,u)- c(z,k,u)|^2}{2|x-z|}\right)G'(|x-z|) \nonumber\\
&\leq  \left( \frac{|c(x,k,u)- c(z,k,u)|^2}{2|x-z|}\right)G'(|x-z|).
\end{align*}
Then we have \begin{equation}
 \label{Sjump}
 \wdt{\varOmega}_{\mathrm j}^{(k)}G(|x-z|) 
  % = \int_{U} (S (|x+c(x,k,u)- z-c(z,k,u)|) - S(|x-z|) -\frac{S'(|x-z|)}{|x-z|}\langle x-z, c(x,k,u) -c(z,k,u)\rangle) \nu(du) 
  \le G'(|x-z|) \int_{U}     \frac{|c(x,k,u)- c(z,k,u)|^2}{2|x-z|} \nu(du). 
\end{equation}
From (\ref{Sdrift}) and (\ref{Sjump}), we see that
\begin{align*}% \label{eq:str-Fe-LkG-computation-nd}  
\wdh \LL_{k} G(|x-z|) 
  & =  [\wdh{\varOmega}_{\mathrm d}^{(k)} + \wdt{\varOmega}_{\mathrm j}^{(k)}]G(|x-z|)\\
%&\ \leq -2\lambda_{R} + G'(|x-z|)\left(-\kappa_{R}g(|x-z|) + \frac{||\sigma_{\lambda_{R}}(x,i)-\sigma_{\lambda_{R}}(z,j)||^2 + 2B(x,k,z,k)}{2|x-z|}\right)\\
% & \ \quad +   G'(|x-z|)\int_{U}\frac{|c(x,k,u)- c(z,k,u)|^2}{2|x-z|}v(du)\\
 & \le -2\lambda_{R} + G'(|x-z|)\bigg(-\kappa_{R}g(|x-z|) + \frac{|\sigma_{\lambda_{R}}(x,i)-\sigma_{\lambda_{R}}(z,j)|^2 + 2B(x,k,z,k)}{2|x-z|}\\
 &   \qquad \qquad\qquad\qquad\qquad\quad+ \int_{U}\frac{|c(x,k,u)- c(z,k,u)|^2}{2|x-z|}\nu(du)\bigg)\\
 & \leq  -2\lambda_{R}.
\end{align*}	
This gives \eqref{LS<beta} when $d \ge 2$. 

% \begin{rem}\label{rem-str-feller-1d}\footnote{We did not use this remark at all. It seems to me that when $d=1$, condition \eqref{eq:str-Fe-coeff-cts} in Assumption \ref{Assum3} can be further weakened to \begin{equation} \label{eq:str-Fe-1d-coeff-cts} 2\langle x-z,b(x,k)-b(z,k)\rangle  + \int_{U}|c(x,k,u)-c(z,k,u)|^2\nu(du) \leq 2\kappa_{R}|x-z|g(|x-z|). \end{equation} Indeed, when \eqref{eq:str-Fe-1d-coeff-cts} hods, your calculations in this remark says that we still have the assertion $\wdh\LL_{k}G(|x-z|) \leq -\beta_{R}$ of Lemma \ref{lem-LkG-estimate}, which, in turn, leads to the strong Feller property. Please check these carefully. Please also check if the proofs in Section 4!}
Finally we look at the case when $d=1$. First we notice that 
% \begin{eqnarray}\bar{A}(x,i,z,j) = \tr A(x,i,z,j)= (\sigma_{\lambda_{R}}(x,i) - \sigma_{\lambda_{R}}(z,j))^2 + 4\lambda_{R}. \end{eqnarray} and hence
\begin{displaymath}
    \bar{A}(x,i,z,j) = \tr A(x,i,z,j)= (\sigma_{\lambda_{R}}(x,i) - \sigma_{\lambda_{R}}(z,j))^2 + 4\lambda_{R}.                %\tr A(x,i,z,j) - \bar{A}(x,i,z,j) + 2B(x,i,z,j) = 2B(x,i,z,j).
\end{displaymath} Using this observation  in \eqref{Sdrift} gives us
 \begin{align} \label{eq:1d-strFe-computation}
&  \wdh{\varOmega}_{\mathrm d}^{(k)}G(|x-z|) \nonumber\\
 &\ \ = \frac{G''(|x-z|)}{2}\bar{A}(x,k,z,k) + \frac{G'(|x-z|)}{2|x-z|}[\tr A(x,k,z,k) - \bar{A}(x,k,z,k) +2B(x,k,z,k)] \nonumber\\
&\ \ \le -2\lambda_{R} + \left(-\kappa_{R}g(|x-z|) + \frac{ 2B(x,k,z,k)}{2|x-z|}\right)G'(|x-z|).
\end{align} The estimation for $\wdt{\varOmega}_{\mathrm j}^{(k)}G(|x-z|)$ is the same as before.    It then follows from  \eqref{Sjump},  \eqref{eq:1d-strFe-computation},  and \eqref{eq:1d-str-Fe-coeff-cts} that \begin{align*} 
   \wdh \LL_{k} G(|x-z|)&  \le   -2\lambda_{R} \\ &\ \ + G'(|x-z|)\bigg(-\kappa_{R}g(|x-z|) + \frac{ 2B(x,k,z,k) +\int_{U} |c(x,k,u)- c(z,k,u)|^2 \nu(du) }{2|x-z|}\bigg)\\& \le -2\lambda_{R},\end{align*} again establishing \eqref{LS<beta} for the case when $d =1$.  The proof is complete. 
 \end{proof}	

\begin{proof}[Proof of Lemma \ref{P^K irredicible}] Let $T>0, r>0$ and $x, a \in \mathbb{R}^d$ be arbitrary but fixed. We will show that $$P^{(k)}(T,x,B(a;r)) = \mathbb{P}\{|X^{(k)}(T)-a| < r| X^{(k)}(0)=x\} > 0$$
	 or equivalently $\mathbb{P}\{|X^{(k)}(T)-a| \geq r| X^{(k)}(0)=x\} < 1$. Let us choose some $t_0 \in (0,T)$. For any $n \in \mathbb{N}$, we set $X^{(k)}_{n}(t_0) := X^{(k)}(t_0)1_{\{|X^{(k)}(t_0)|\leq n\}}$. Since $\lim_{r\rightarrow 0} F(r) = 0$ and $0 \leq F \leq 1$, the bounded convergence implies that
\begin{equation}\label{eq:E FXnk-0}
\lim\limits_{n\rightarrow \infty}\mathbb{E}[F(|X^{(k)}_{n}(t_0) - X^{(k)}(t_0)|^2)] = 0.
\end{equation}
For $t \in [t_0, T]$, define
$$J^n(t) := \frac{T-t}{T-t_0}X^{(k)}_{n}(t_0) +                   \frac{t-t_0}{T-t_0}a, \ \text{ and }\ h^n(t):= \frac{a-X^{(k)}_{n}(t_0)}{T-t_0}-b(J^n(t),k).$$
Observe  that $J^n(t_0) = X^{(k)}_{n}(t_0)$ and $J^n(T) = a$. In addition,   $J^n$  satisfies the following  stochastic differential equation
\begin{align*}
J^n(t) = X^{(k)}_{n}(t_0) + \int_{t_0}^{t}b(J^n(s),k)ds + \int_{t_0}^{t}h^n(s)ds, ~~~~ t \in [t_0,T].
\end{align*}

Consider the  stochastic differential equation
\begin{equation}
\label{Y-t0-T sde}
\begin{aligned}
Y(t)&  = X^{(k)}(t_0) + \int_{t_0}^{t}[b(Y(s),k) + h^n(s)]ds + \int_{t_0}^{t}\sigma(Y(s),k)dW(s) 
\\& \qquad + \int_{t_0}^{t}\int_{U}c(Y(s),k,u)\tilde{N}(ds,du), \qquad t \in [t_0,T].
\end{aligned}
\end{equation}
Also denote $\Delta_t := Y(t) - J^n(t)$ for $t \in [t_0,T]$. Note that $\Delta_{t_0} = X^{(k)}(t_0) - X^{(k)}_n(t_0)$ and $\Delta_{T} = Y(T)-a$. In addition,  $\Delta_{t}$ satisfies the stochastic differential equation
\begin{align*}
\Delta_t = \Delta_{t_0}& + \int_{t_0}^{t}[b(Y(s),k) - b(J^n(s),k)]ds + \int_{t_0}^{t}\sigma(Y(s),k)dW(s)   + \int_{t_0}^{t}\int_{U}c(Y(s),k,u)\tilde{N}(ds,du).
\end{align*} 
 Consequently  the generator of the process $\Delta_t$ is given by
\begin{align*}
\LL  f(x) &=  \LL_{\mathrm d} f(x) + \LL_{\mathrm j}f(x)\\
  :& = \frac{1}{2}\tr\left(\sigma(Y(s),k)\sigma(Y(s),k)^T\nabla^{2}f(x)\right) + \langle b(Y(s),k)-b(J^n(s),k),\nabla f(x)\rangle \nonumber\\
&\qquad + \int_{U}\left(f(x+c(Y(s),k,u)) - f(x) -\langle\nabla f(x), c(Y(s),k,u)\rangle\right)\nu(du), \quad f \in C^{2}_{c}(\R^{d}). 
\end{align*}

We compute
% From (\ref{F'>0}) and (\ref{F''<0}), we obtain 
\begin{align*}
 \LL_{\mathrm d}  F(|\Delta_s|^2) &= \frac{1}{2}\tr\left(\sigma(Y(s),k)\sigma(Y(s),k)^T\nabla^{2}F(|\Delta_s|^2) \right) + \langle b(Y(s),k)-b(J^n(s),k),\nabla F(|\Delta_s|^2) \rangle \\
&= \frac{1}{2}\tr\left(\sigma(Y(s),k)\sigma(Y(s),k)^T\left[4F''(|\Delta_s|^2)\Delta_s\Delta_s^T + 2F'(|\Delta_s|^2)I\right] \right)\\
&  \quad +  \langle b(Y(s),k)-b(J^n(s),k),2F'(|\Delta_s|)\Delta_s \rangle \\
% &= 2F''(|\Delta_s|^2)tr\left(\sigma(Y(s),k)\sigma(Y(s),k)^T\Delta_s\Delta_s^T\right) + F'(|\Delta_s|^2)tr\left(\sigma(Y(s),k)\sigma(Y(s),k)^T\right) \\
% & \quad+  2F'(|\Delta_s|^2)\langle b(Y(s),k)-b(J^n(s),k),\Delta_s \rangle\\
&= 2F''(|\Delta_s|^2)|\Delta_s^T\sigma(Y(s),k)|^2 + F'(|\Delta_s|)|\sigma(Y(s),k)|^2\\
& \quad +  2F'(|\Delta_s|^2)\langle b(Y(s),k)-b(J^n(s),k),\Delta_s \rangle\\
&\leq F'(|\Delta_s|)\left[|\sigma(Y(s),k)|^2 + 2\langle b(Y(s),k)-b(J^n(s),k),\Delta_s \rangle\right],
 %\\ &{\blue \leq |\sigma(Y(s),k)|^2 + 2\langle b(Y(s),k)-b(J^n(s),k),\Delta_s \rangle,}
\end{align*} %\footnote{There are some problem here. If $|\sigma(Y(s),k)|^2 + 2\langle b(Y(s),k)-b(J^n(s),k),\Delta_s \rangle< 0$, then the last inequality may not hold. I think the arguments below fixed this problem.} 
where the inequality follows from %(\ref{F'>0}) and 
(\ref{F''<0}).
Likewise, the concavity of $F$ leads to
\begin{align*}
 \LL_{\mathrm j} F(|\Delta_s|^2) &= \int_{U}\left(F(|\Delta_s+c(Y(s),k,u)|^2) - F(|\Delta_s|^2) -\langle\nabla F(|\Delta_s|^2), c(Y(s),k,u)\rangle\right)\nu(du)\\
 &\le  \int_{U}  \!\big[F'(|\Delta_s|^2)[|\Delta_s+c(Y(s),k,u)|^2-|\Delta_s|^2] - 2F'(|\Delta_s|^2)\langle\Delta_s, c(Y(s),k,u)\rangle\big] \!\nu(du)\\
 & = \int_{U}  F'(|\Delta_s|^2)|c(Y(s),k,u)|^2  \nu(du). %  \\ & \le  \int_{U}  |c(Y(s),k,u)|^2  \nu(du).
\end{align*}
 Therefore, by adding the above two inequalities, we have
\begin{align*}
\LL F(|\Delta_s|^2) &
\leq  { F'(|\Delta_s|^2)\bigg[  |\sigma(Y(s),k)|^2 + 2\langle b(Y(s),k)-b(J^n(s),k),\Delta_s \rangle + \int_{U}|c(Y(s),k,u)|^2\nu(du)\bigg]. }
% \\ &\leq  \kappa(|Y(s)|^2 + 1) + 2|\Delta_s|g(|\Delta_s|);
\end{align*} 
  On the other hand, when  $|Y(s) | \le R$, $|J^{n}(s) |\le R$ and $|\Delta_{s}| \le \delta_{0}$,  we can use  \eqref{eq:str-Fe-coeff-cts} and  \eqref{< |x|^2+1}  to obtain
 % \footnote{A possible fix for this problem is to assume \begin{align*}| \lan b(x,k), x\ran | + |\sigma(x,k)|^{2} + \int_{U} |c(x,k,u)|^{2} \nu(du) \le K(1+|x|^{2})\\\intertext{ or }\lan b(x,k), x\ran  \le K(1+|x|^{2})  \text{ and }  |\sigma(x,k)|^{2} + \int_{U} |c(x,k,u)|^{2} \nu(du) \le K(1+|x|^{2})\end{align*} in \eqref{< |x|^2+1}. Alternatively, we can assume that $b(\cdot, k)$ is locally bounded. These are not the best condition, but still common in the textbooks.  }
\begin{align*} 
 |\sigma (Y(s),k)|^2   + 2\langle b(Y(s),k)& -b(J^n(s),k),\Delta_s \rangle  + \int_{U}|c(Y(s),k,u)|^2\nu(du)      \\
  & \le  \kappa(|Y(s)|^2 + 1) + 2\kappa_{R} |\Delta_s | g(|\Delta_s|) \le K_{R} + \kappa |Y(s)|^{2},  
\end{align*}
% \begin{displaymath}
%\LL F(|\Delta_s|^2)
% \le  \kappa(|Y(s)|^2 + 1) + 2\kappa_{R} |\Delta_s | g(|\Delta_s|) \le K_{R} + \kappa |Y(s)|^{2},  
%\end{displaymath}
 where $K_{R} = \kappa + 2\kappa_{R} \max_{r\in [0,\delta_{0}]}  \{ r  g(r)\} < \infty.$
 % note that we used \eqref{< |x|^2+1} and \eqref{eq:str-Fe-coeff-cts} to obtain the last inequality.\footnote{Note that the last inequality holds if $|Y(s) | \le R$, $|J^{n}(s) |\le R$ and $|\Delta_{s}| \le \delta_{0}$.}
%%%%%%%%%%%%%%%%%%%%%%%%%%%%%%%%%%%%%%%%%%%%%%%%%%%%%%%%%%%%%
% From (\ref{E[X^2] < K}), we know that $\mathbb{E}[|Y(s)|^2] \leq K$ for all $s \in [0,T]$, 
Then it follows that \begin{align*} 
 \LL F(|\Delta_s|^2) 
  \le    F'(|\Delta_s|^2) [ K_{R} + \kappa |Y(s)|^{2}]  \le  K_{R} + \kappa |Y(s)|^{2}.\end{align*}
  In view of \eqref{< |x|^2+1} and \eqref{E[X^2] < K}, we can use the standard arguments to show that $\E[\sup_{t_{0}\le s \le T} |Y(s)|^{2}] \le K $,  where $K$ is a positive constant independent of $t_0$. 
 % Then we have \begin{eqnarray} \mathscr A F(|\Delta_s|^2) &\leq&  \kappa(K^2 + 1) + 2|\Delta_s|g(|\Delta_s|).\end{eqnarray}
%%%%%%%%%%%%%%%%%%%%%%%%%%%%%%%%%%%%%%%%%%%%%%%%%%%%%%%%%%%%%%
For any $R>0$, we define $\tau_{R} := \inf\{t\geq t_0 : |Y(t)|\vee |J^n(t)| > R\} \wedge T$ and $S_{\delta_{0}} :=\inf\{t\geq t_0 : |Y(t) - J^n(t)| \geq \delta_{0}\} \wedge T$. Then  we can compute
\begin{align} \label{E[F] < E[F] +0}
\nonumber\mathbb{E}[F(|\Delta_{T\wedge\tau_{R}\wedge S_{\delta_{0}}}|^2)]
 &= \mathbb{E}[F(|\Delta_{t_0}|^2)] + \mathbb{E}\left[\int_{t_0}^{T\wedge\tau_{R}\wedge S_{\delta_{0}}}\LL F(|\Delta_{s^{-}}|^2)ds\right] \\
\nonumber & \le \mathbb{E}[F(|\Delta_{t_0}|^2)] + \mathbb{E}\left[\int_{t_0}^{T\wedge\tau_{R}\wedge S_{\delta_{0}}}  ( K_{R} + \kappa |Y(s^{-})|^{2}) ds\right] \\
\nonumber & \le \E[F(|\Delta_{t_0}|^2)] + K_{R} (T-t_{0}) + \E\bigg[\int_{t_{0}}^{T} \kappa |Y(s)|^{2} ds  \bigg]
 \\ & \le   \E[F(|\Delta_{t_0}|^2)] + ( K_{R} +  \kappa K) (T-t_{0}). 
 \end{align}
% \begin{align}\label{E[F] < E[F] + epsilon}
% \mathbb{E}[F(|\Delta_{T\wedge\tau_{R}\wedge S_{\delta_{0}}}|^2)]
% &= \mathbb{E}[F(|\Delta_{t_0}|^2)] + \mathbb{E}\left[\int_{t_0}^{T\wedge\tau_{R}\wedge S_{\delta_{0}}}\LL F(|\Delta_s|^2)ds\right] \nonumber\\
% &\leq \mathbb{E}[F(|\Delta_{t_0}|^2)] + \mathbb{E}\left[\int_{t_0}^{T\wedge\tau_{R}\wedge S_{\delta_{0}}}\left(\kappa(K + 1) + 2|\Delta_s|g(|\Delta_s|)\right)ds\right] \nonumber\\
% &\leq \mathbb{E}[F(|\Delta_{t_0}|^2)] + \mathbb{E}\left[\int_{t_0}^{T\wedge\tau_{R}\wedge S_{\delta_{0}}}\left(\kappa(K + 1) + 2\delta_{0}g(|\Delta_s|)\right)ds\right] \nonumber\\
% &\leq \mathbb{E}[F(|\Delta_{t_0}|^2)] + \mathbb{E}\left[\int_{t_0}^{T}\left(\kappa(K+ 1) + 2\delta_{0}g(|\Delta_s|)\right)ds\right] \nonumber\\
%&= \mathbb{E}[F(|\Delta_{t_0}|^2)] +  \kappa(K + 1)(T-t_{0}) + 2\delta_{0}\mathbb{E}\left[\int_{t_0}^{T} g(|\Delta_s|)ds\right]. 
%\end{align}
%{\blue Since $\lim_{R\rightarrow \infty}\tau_{R} = \infty$ a.s. and $0\le F\le 1$, passing to the limit as $R\to \infty$ gives \begin{equation}\label{E[F] < E[F] +0}
%\E[F(|\Delta_{T\wedge S_{\delta_{0}}}|^2)]\leq \mathbb{E}[F(|\Delta_{t_0}|^2)] +  \kappa(K + 1)(T-t_{0}) + 2\delta_{0}\mathbb{E}\left[\int_{t_0}^{T} g(|\Delta_s|)ds\right].
%\end{equation}}

% The same techniques used in \citet*{Jump type}
Next we  show that 
\begin{eqnarray}\label{E[F] < 1/FE}
\mathbb{E}[F(|\Delta_{T}|^2)] \leq \frac{1}{F(\delta_{0}^2)}\mathbb{E}[F(|\Delta_{T\wedge S_{\delta_{0}}}|^2)].
\end{eqnarray}  
To this end, we note that $|\Delta_{T\wedge S_{\delta_{0}}\wedge\tau_{R}}| \geq \delta_{0}$ on the set  $\{S_{\delta_{0}} < T \wedge\tau_{R}\}$. Since $F$ is increasing, we have $F(\delta_{0}^2) \leq F(|\Delta_{T\wedge S_{\delta_{0}}}|^2)$. This together with the fact that $0 \leq F \leq 1$ give the following 
\begin{align*}
& \frac{\mathbb{E}[F(|\Delta_{T\wedge\tau_{R}\wedge S_{\delta_{0}}}|^2)]}{F(\delta_{0}^2)} - \mathbb{E}[F(|\Delta_{T\wedge\tau_{R}}|^2)] \\
&\ \ = \frac{\mathbb{E}[F(|\Delta_{T\wedge\tau_{R}\wedge S_{\delta_{0}}}|^2)1_{\{T\wedge\tau_{R}\leq S_{\delta_{0}}\}}] + \mathbb{E}[F(|\Delta_{T\wedge\tau_{R}\wedge S_{\delta_{0}}}|^2)1_{\{T\wedge\tau_{R}> S_{\delta_{0}}\}}]}{F(\delta_{0}^2)} - \mathbb{E}[F(|\Delta_{T\wedge\tau_{R}}|^2)] \\
&\ \ \geq \frac{\mathbb{E}[F(|\Delta_{T\wedge\tau_{R}}|^2)1_{\{T\wedge\tau_{R}\leq S_{\delta_{0}}\}}] + F(\delta_{0}^2)\mathbb{P}\{T\wedge\tau_{R}> S_{\delta_{0}}\}}{F(\delta_{0}^2)} - \mathbb{E}[F(|\Delta_{T\wedge\tau_{R}}|^2)] \\
&\ \ \geq \mathbb{P}\{T\wedge\tau_{R}> S_{\delta_{0}}\} + \mathbb{E}[F(|\Delta_{T\wedge\tau_{R}}|^2)1_{\{T\wedge\tau_{R}\leq S_{\delta_{0}}\}}]  -\mathbb{E}[F(|\Delta_{T\wedge\tau_{R}}|^2)] \\
&\ \ = \mathbb{P}\{T\wedge\tau_{R}> S_{\delta_{0}}\} - \mathbb{E}[F(|\Delta_{T\wedge\tau_{R}}|^2)1_{\{T\wedge\tau_{R}> S_{\delta_{0}}\}}] \\
&\ \ \geq \mathbb{P}\{T\wedge\tau_{R}> S_{\delta_{0}}\} - \mathbb{E}[1\cdot1_{\{T\wedge\tau_{R}> S_{\delta_{0}}\}}]  = 0.
\end{align*} 
Consequently we have $\mathbb{E}[F(|\Delta_{T\wedge\tau_{R}}|^2)] \leq \frac{\E[F(|\Delta_{T\wedge\tau_{R}\wedge S_{\delta_{0}}}|^2)]}{F(\delta_{0}^2)} $.
Since $\lim_{R\rightarrow \infty}\tau_{R} = T$ a.s. and $0 \leq F \leq 1$, the bounded convergence theorem gives (\ref{E[F] < 1/FE}).

% Next, we proceed as in \citet*{Jump type} and \citet*{Qiao}.
Recall that $Y$ satisfies the stochastic differential equation \eqref{Y-t0-T sde} for $t\in [t_{0}, T]$. For $t \in [0,t_0]$, we define $Y(t) := X^{(k)}(t)$ and $X^{(k)}(t)$ is the weak solution to \eqref{SDE X^k} with initial condition $x$. Then the process $Y$ satisfies the following  stochastic differential equation:
\begin{align*}
Y(t) = x&  + \int_{0}^{t}[b(Y(s),k) + h^n(s)1_{\{s>t_0\}}]ds + \int_{0}^{t}\sigma(Y(s),k)dW(s) + \int_{0}^{t}\int_{U}c(Y(s),k,u)\tilde{N}(ds,du)
\end{align*}
for $t \in [0,T]$. Next we set
\begin{align*}
H(t) := 1_{\{t>t_0\}}\sigma^{-1}(Y(t),k)h^n(t),\ \text{ and }\
M(t) := \exp\bigg\{\int_{0}^{t}\langle H(s), dW(s) \rangle -\frac{1}{2}\int_{0}^{t}|H(s)|^2ds\bigg\}.
\end{align*}
As argued in \citet*{Qiao-14}, { it follows from \eqref{eq1:elliptic}   that $|H(t)|^2$ is bounded and hence $M$ is a martingale under $\mathbb{P}$ by Novikov's criteria}. Moreover, $\mathbb{E}[M(T)] = 1$. Define
\begin{align*}
\mathbb{Q}(B) &:= \mathbb{E}[M(T)1_{\{B\}}], ~~~B \in \mathcal{F}_T\\
\tilde{W}(t) &:= W(t) + \int_{0}^{t}H(s)ds.
\end{align*}
It follows from Theorem 132 of \citet*{Situ-05} that $\mathbb{Q}$ is a probability measure, $\tilde{W}$ is a $\mathbb{Q}$-Brownian motion and $\tilde{N}(dt,du)$ is a $\mathbb{Q}$-compensated Poisson random  measure with compensator $dt\nu(du)$. Furthermore, under the measure $\mathbb{Q}$, $Y$ solves the following stochastic differential equation
\begin{align*}
	Y(t) = x + \int_{0}^{t}b(Y(s),k)ds + \int_{0}^{t}\sigma(Y(s),k)d\tilde{W}(s) + \int_{0}^{t}\int_{U}c(Y(s),k,u)\tilde{N}(ds,du)
\end{align*}
for $t \in [0,T]$. By the uniqueness in law of the solution to the SDE, we have that the law of $\{X^{(k)}(t): t \in [0,T]\}$ under $\mathbb{P}$ is the same as the law of  $\{Y(t): t \in [0,T]\}$ under $\mathbb{Q}$. In particular, we have   $\mathbb{P}\{|X^{(k)}(T)-a| \geq r| X^{(k)}(0)=x\}  =\mathbb{Q}\{|Y(T)-a| \geq r| Y(0)=x\}$.  Since  $\mathbb{P}$ and  $\mathbb{Q}$ are equivalent, 
  the desired assertion $\P\{|X^{(k)}(T) - a| \ge r | X^{(k)}(0)=x\}= \mathbb{Q}\{|Y(T)-a| \geq r| Y(0)=x\}<1 $ will follow if we can show that $\P\{|Y(T) -a| \ge r | Y(0)=x\} <1$.
 % Moreover, (\ref{E[F] < 1/FE}) and (\ref{E[F] < E[F] +0}) hold under the measure $\mathbb{Q}$. Since $F$ is increasing we obtain
 To this end, for any $\e > 0$, we first choose an $R> 0$ sufficiently large so that $\P\{\tau_{R} < T \} < \e$.  Sine $F$ is bounded and increasing,  we can use (\ref{E[F] < 1/FE}) and (\ref{E[F] < E[F] +0}) to compute
\begin{align*}
\P\{|Y(T) -a| \ge r | Y(0)=x\}& = \P \{|Y(T)-a|^2 \geq r^2| Y(0)=x\}\\
&=\P\{F(|Y(T)-a|^2) \geq F(r^2)| Y(0)=x\}\\
&\leq \frac{\mathbb{E}[F(|Y(T)-a|^2)]}{F(r^2)}\\
& = \frac{\mathbb{E}[F(|\Delta_{T}|^2)]}{F(r^2)}\\
&\leq \frac{\mathbb{E}[F(|\Delta_{T\wedge S_{\delta_{0}}}|^2)]}{F(r^2)F(\delta_{0}^2)}\\
& = \frac{\mathbb{E}[F(|\Delta_{T\wedge S_{\delta_{0}}\wedge \tau_{R}}|^2)1_{\{\tau_{R} \ge T\wedge S_{\delta_{0}}\}}] + \E[F(|\Delta_{T\wedge S_{\delta_{0}}}|^2)1_{\{\tau_{R} < T\wedge S_{\delta_{0}}\}}]}{F(r^2)F(\delta_{0}^2)}\\ 
& \le \frac{\E[F(|\Delta_{t_0}|^2)] + ( K_{R} + \kappa K) (T-t_{0}) + \P\{\tau_{R} < T \}}{F(r^2)F(\delta_{0}^2)}\\
& \le \frac{\E[F(|\Delta_{t_0}|^2)] + ( K_{R} + \kappa K) (T-t_{0}) + \e}{F(r^2)F(\delta_{0}^2)}.
% &\leq \frac{\mathbb{E}[F(|\Delta_{t_0}|^2)] +  \kappa(K^2 + 1)(T-t_{0}) + 2\delta_{0}\mathbb{E}[\int_{t_0}^{T} g(|\Delta_s|)ds] }{F(r^2)F(\delta_{0}^2)}.
\end{align*}
 Thanks to \eqref{eq:E FXnk-0},  we have $\E[F(|\Delta_{t_0}|^2)] \to 0$ as $n\to \infty$.
Therefore  we can choose $n$ sufficiently large and $t_0$ close enough to $T$ to make the last term less than $1$ as desired.\end{proof}

 %\bibliographystyle{newapa}
%\bibliography{refs}% \bibliography{C:/Users/Sony/Desktop/Papers/refs}\end{document} 

\end{document}